\documentclass[11pt]{article}

\usepackage{latexsym}
\usepackage{amssymb}
\usepackage{amsthm}
\usepackage{amscd}
\usepackage{amsmath}
\usepackage[all]{xy}
\usepackage{graphicx}

\newcounter{casecount}

\newtheorem{theorem}{Theorem}[section]

\newtheorem{case}[casecount]{Case}

\newtheorem{lemma}{Lemma}[section]

\newtheorem{proposition}{Proposition}[section]
\newtheorem{corollary}{Corollary}[section]

\xyoption{dvips}

\CompileMatrices

\numberwithin{equation}{section}

\newcommand{\FF}{\mathbb{F}}    \newcommand{\CC}{\mathbb{C}}    
 \def\dim{\mathrm{dim}}  
 \def\ZZ{\mathbb{Z}} 
       \def\Res{\mathrm{Res}} \def\dd{\displaystyle}   
  
  \def\wt{\mathrm{wt}}
  
  \def\sss{\scriptscriptstyle} \def\ss{\scriptstyle}
\newcommand{\cT}{\mathcal{T}}

\newcommand{\rank}{\mathrm{rank}}

\newcommand{\fkn}{\mathfrak{n}}
\newcommand{\supp}{\mathrm{supp}}

\newcommand{\cP}{\mathcal{P}}

\newcommand{\cS}{\mathcal{S}}
\newcommand{\Cols}{\mathrm{Cols}}
\newcommand{\Rows}{\mathrm{Rows}}
\newcommand{\One}{{1\hspace{-.14cm} 1}}

\newcommand{\Null}{\mathrm{Null}}

\newcommand{\cZ}{\mathcal{Z}}
\newcommand{\cend}{\mathrm{lc}}
\newcommand{\br}{\mathrm{br}}
\newcommand{\larc}[1]{\hspace{-.4ex}\overset{#1}{\frown}\hspace{-.4ex}}
\newcommand{\bag}{\mathrm{bag}}
\newcommand{\lc}{\mathrm{lc}}

\makeatletter
\renewcommand{\@makefnmark}{\mbox{\textsuperscript{}}}
\makeatother

\allowdisplaybreaks[1]

\UseCrayolaColors

\def\adots{\mathinner{\mkern2mu\raise0pt\hbox{.}  % antidiagonal dots
\mkern2mu\raise4pt\hbox{.}\mkern1mu
\raise7pt\vbox{\kern7pt\hbox{.}}\mkern1mu}}

\begin{document}

\title{Restricting supercharacters of the finite group\\ of unipotent uppertriangular matrices}
\author{Nathaniel Thiem\footnote{University of Colorado at Boulder: \textsf{thiemn@colorado.edu}}{ } and Vidya Venkateswaran\footnote{University of Chicago: \textsf{vidyav@math.uchicago.edu}}}
\date{}

\maketitle

\abstract{It is well-known that the representation theory of the finite group of unipotent 
upper-triangular matrices $U_n$ over a finite field is a wild problem.  By instead 
considering approximately irreducible representations (supercharacters), one obtains a rich combinatorial 
theory analogous to that of the symmetric group, where we replace partition combinatorics 
with set-partitions.  This paper studies the supercharacter theory of a family of subgroups that interpolate between $U_{n-1}$ and $U_n$.  We supply several combinatorial indexing sets for the supercharacters, supercharacter formulas for these indexing sets, and a combinatorial rule for restricting supercharacters from one group to another.  A consequence of this analysis is a Pieri-like restriction rule from $U_n$ to $U_{n-1}$ that can be described on set-partitions (analogous to the corresponding symmetric group rule on partitions). }

\section{Introduction}

The representation theory of the finite group of upper-triangular matrices $U_n$ is  a well-known wild problem.  Therefore, it came as somewhat of a surprise when C. Andr\'e was able to show that by merely ``clumping" together some of the conjugacy classes and some of the irreducible representations one attains a workable approximation to the representation theory of $U_n$ \cite{An95,An99,An01,An02}. In his Ph.D. thesis \cite{Ya01}, N. Yan showed how the algebraic geometry of the original construction could be replaced by more elementary constructions.   E. Arias-Castro, P. Diaconis, and R. Stanley \cite{ADS04} were then able to demonstrate that this theory can in fact be used to study random walks on $U_n$ using techniques that traditionally required the knowledge of the full character theory \cite{DS93}.  Thus, the approximation is fine enough to be useful, but coarse enough to be computable.  Furthermore, this approximation has a remarkable combinatorial structure analogous to that of the symmetric group, where we replace partitions with set-partitions,
$$\left\{\begin{array}{c} \text{Almost irreducible}\\ \text{representations of $U_n$}\end{array}\right\} \longleftrightarrow \left\{\begin{array}{c} \text{Labeled set partitions}\\ \text{of $\{1,2,\ldots,n\}$}\end{array}\right\}.$$
 One of the main results of this paper is to extend the analogy with the symmetric group by giving a combinatorial Pieri-like formula for set-partitions that corresponds to restriction in $U_n$.  

In \cite{DI06}, P. Diaconis and M. Isaacs generalized this approximating approach to develop a the concept of a \emph{supercharacter theory} for all finite groups, where irreducible characters are replaced by supercharacters and conjugacy classes are replaced by superclasses.  In particular, their paper generalized Andr\'e's original construction by giving an example of a supercharacter theory for a family of groups called algebra groups.   For this family of groups, they show that supercharacters restrict to $\ZZ_{\geq 0}$-linear combination of supercharacters, tensor products of supercharacters are $\ZZ_{\geq 0}$-linear combinations of supercharacters, and they develop a notion of superinduction that is the adjoint functor to restriction for supercharacters.   This paper uses a family of subgroups that interpolate between $U_n$ and $U_{n-1}$ to explicitly decompose restricted supercharacters from $U_n$ to $U_{n-1}$.

Section 2 reviews the basics of supercharacter theory and pattern groups.  Section 3 defines the interpolating subgroups  $U_{(m)}$, and finds two different sets of natural superclass and supercharacter representatives, which we call comb representatives and path representatives.  Section 4 uses a general character formula from \cite{DT07} to determine character formulas for both comb and path representatives.  The character formula for comb representatives -- Theorem \ref{TriangleCharacterFormula} -- is easier to compute directly, but the path representative character formula -- Theorem \ref{PathCharacterFormula} -- has a more pleasing combinatorial structure.    Section 5 uses the character formulas to derive a restriction rule for the interpolating subgroups given in Theorem \ref{InterpolatingRestriction}.   Corollary \ref{URestrictionRule} iterates these restrictions to deduce a recursive decomposition formula for the restriction from $U_n$ to $U_{n-1}$.  

This paper is the companion paper to \cite{MT07}, which studies the superinduction of supercharacters.  Other work related to supercharacter theory of unipotent groups, include C. Andr\'e and A. Neto's  exploration of supercharacter theories for unipotent groups of Lie types $B$, $C$, and $D$ \cite{AN06}, C. Andr\'e and  A. Nicol\'as' analysis of supertheories over other rings \cite{AnNi06}, and an intriguing possible connection between supercharacter theories and  Boyarchenko and Drinfeld's work on $L$-packets \cite{BD06}.

\subsubsection*{Acknowledgements}

Part of this work is Venkateswaran's honors thesis at Stanford University.  We would like to thank Diaconis and Marberg for many enlightening discussions regarding this work.

\section{Preliminaries}

This section gives reviews several topics fundamental to our main results: Supercharacter theories, pattern groups, and a character formula for pattern groups.

\subsection{Supertheories}

 Let $G$ be a group.  A \emph{supercharacter theory} for $G$ is a partition $\cS^\vee$ of the elements of $G$ and a set of characters $\cS$, such that
\begin{enumerate}
\item[(a)] $|\cS|=|\cS^\vee|$,
\item[(b)] Each $S\in \cS^\vee$ is a union of conjugacy classes,
\item[(c)] For each irreducible character $\gamma$ of $G$, there exists a unique $\chi\in \cS$ such that
$$\langle \gamma, \chi\rangle>0,$$
where $\langle,\rangle$ is the usual innerproduct on class functions,
\item[(d)] Every $\chi\in \cS$ is constant on the elements of $\cS^\vee$.
\end{enumerate}
We call $\cS^\vee$ the set of \emph{superclasses} and $\cS$ the set of \emph{supercharacters}.  Note that every group has two trivial supercharacter theories -- the usual character theory and the supercharacter theory with $\cS^\vee=\{\{1\},G\setminus\{1\}\}$ and $\cS=\{\One,\gamma_G-\One\}$, where $\One$ is the trivial character of $G$ and $\gamma_G$ is the regular character.

There are many ways to construct supercharacter theories, but this paper will study a particular version developed in \cite{DI06} to generalize Andr\'e's original construction to a larger family of groups called algebra groups.

\subsection{Pattern groups}

While many results can be stated in the generality of algebra groups, frequently statements become simpler if we restrict our attention to a subfamily called pattern groups.

Let $U_n$ denote the set of $n\times n$ unipotent upper-triangular matrices with entries in the finite field $\FF_q$ of $q$ elements.   For any poset $\cP$ on the set $\{1,2,\ldots, n\}$, the pattern group $U_\cP$ is given by
$$U_\cP=\{u\in U_n\ \mid\ u_{ij}\neq 0\text { implies $i<j$ in $\cP$}\}.$$
This family of groups includes unipotent radicals of rational parabolic subgroups of the finite general linear groups, and $U_n$ is the pattern group corresponding to the total order $1<2<3<\cdots<n$.  

The group $U_\cP$ acts on the $\FF_q$-algebra
$$\fkn_\cP=\{u-1\ \mid\ u\in U_\cP\}$$
by left and right multiplication.  Two elements $u,v\in U_\cP$ are in the same \emph{superclass} if $u-1$ and $v-1$ are in the same two-sided orbit of $\fkn_\cP$.  Note that since every element of $U_\cP$ can be decomposed as a product of elementary matrices, every element in the orbit containing $v-1\in\fkn_\cP$ can be obtained by applying a sequence of the following row and column operations.
\begin{enumerate}
\item[(a)] A scalar multiple of row $j$ may be added to row $i$ if $j>i$ in $\cP$,
\item[(b)] A scalar multiple of column $k$ may be added to column $l$ if $k<l$ in $\cP$.
\end{enumerate}

There are also left and right actions of $U_\cP$ on the dual space
$$\fkn_\cP^*=\{\lambda:\fkn_\cP\rightarrow \FF_q\ \mid\ \lambda\text{ $\FF_q$-linear}\},$$
given by
$$(u\lambda v)(x-1)=\lambda(x^{-1}(x-1)y^{-1}), \qquad\text{where $\lambda\in \fkn_\cP^*$, $u,v,x\in U_\cP$.}$$
Fix a nontrivial group homomorphism $\theta:\FF_q^+\rightarrow \CC^\times$.  The \emph{supercharacter} $\chi^\lambda$ with representative $\lambda\in \fkn_\cP^*$ is
$$\chi^\lambda=\frac{|U_\cP \lambda|}{|U_\cP \lambda U_\cP|}\sum_{\mu\in U_\cP \lambda U_\cP} \theta\circ(-\mu).$$
We many identify  the functions $\lambda\in \fkn_\cP^*$ with matrices by the convention
$$\lambda_{ij}=\lambda(e_{ij}),\qquad \text{where $e_{ij}\in \fkn_\cP$ has $(i,j)$ entry 1 and zeroes elsewhere.}$$
Then, as with superclasses, every element in the orbit containing $\lambda\in\fkn_\cP^*$ can be obtained by applying a sequence of the following row and column operations.
\begin{enumerate}
\item[(a)] A scalar multiple of row $i$ may be added to row $j$ if $i<j$ in $\cP$,
\item[(b)] A scalar multiple of column $l$ may be added to column $k$ if $l>k$ in $\cP$.
\end{enumerate}
Note that we ignore (or set to zero) all nonzero elements that might occur through these operations that are not in allowable in $\fkn_\cP^*$.  For example, if $i$ is incomparable to $j$ in $\cP$ and a row operation would cause $\lambda_{ij}\neq 0$, we carry through the operation as usual, and retroactively set $\lambda_{ij}$ back to zero (since $\lambda$ does not ``see" these entries).

\vspace{.25cm}

\noindent \textbf{Example.}  For $U_n$ we have
\begin{equation} \label{UOrbitRepresentatives}
\left\{\begin{array}{c} \text{Superclasses}\\ \text{of $U_n$}\end{array}\right\}\longleftrightarrow 
\left\{u\in U_n\ \big|\ \begin{array}{c} \text{$u-1$ has at most one nonzero}\\ \text{ element in every and column}\end{array}\right\}
\end{equation}
Note that if $q=2$, then 
$$\left\{u\in U_n\ \big|\ \begin{array}{c} \text{$u-1$ has at most one nonzero}\\ \text{ element in every and column}\end{array}\right\} \longleftrightarrow \left\{\begin{array}{c}\text{Set partitions}\\ \text{of $\{1,2,\ldots,n\}$}\end{array}\right\}.$$
Similarly, if 
\begin{align*}
\fkn_n&=U_n-1\\
\fkn_n^*&= \{\lambda:\fkn_n\rightarrow \FF_q\ \mid\ \text{$\lambda$ $\FF_q$-linear}\}.
\end{align*}
then
\begin{equation} \label{UcoOrbitRepresentatives}
\left\{\begin{array}{c} \text{Supercharacters}\\ \text{of $U_n$}\end{array}\right\}\longleftrightarrow 
\left\{\lambda\in \fkn_n^*\ \big|\ \begin{array}{c} \text{The matrix $\lambda$ has at most one non-}\\ \text{zero element in every and column}\end{array}\right\}
\end{equation}
Let
\begin{equation}
\cS_n(q)=\{\lambda\in \fkn_n^*\ \mid\ \text{$\lambda$ has at most one nonzero element in every and column}\}.
\end{equation}

\subsection{A supercharacter formula for pattern groups}

Let $U_\cP$ be a pattern group with corresponding nilpotent algebra $\fkn_\cP$.   Let
$$J=\{(i,j)\ \mid\ i<j \text{ in $\cP$}\}.$$
Given $u\in U_\cP$ and $\lambda\in \fkn_\cP^*$, define $a,b\in \FF_q^{|J|}$ by
\begin{align*}
a_{ij}&=\sum_{j<k \text{ in $\cP$}} u_{jk}\lambda_{ik}, \qquad \text{for $(i,j)\in J$,}\\
b_{jk}&=\sum_{i<j\text{ in $\cP$}} u_{ij}\lambda_{ik}, \qquad \text{for $(j,k)\in J$.}
\end{align*}
Let $M$ be the $|J|\times |J|$ matrix given by
$$M_{ij,kl} =\left\{\begin{array}{ll} u_{jk}\lambda_{il}, & \text{if $i<j<k<l$ in $\cP$,}\\ 0, & \text{otherwise}.\end{array}\right., \qquad \text{for $(i,j),(k,l)\in J$}.$$
Informally, if one superimposes the matrices $u$ and $\lambda$, then
\begin{align*}
a \qquad &\text{tracks occurrences of}\qquad \begin{array}{|@{}c@{}|} \hline \lambda_{jk} \\ \\ u_{ik}\\ \hline\end{array}\\
b \qquad &\text{tracks occurrences of} \qquad \begin{array}{|@{}ccc@{}|} \hline u_{ij} & & \lambda_{ik}\\ \hline\end{array}\\
M \qquad &\text{tracks occurrences of} \qquad \begin{array}{|@{}ccc@{}|} \hline  & & \lambda_{il}\\ u_{jk} & & \\ \hline\end{array}\\
\end{align*}

The following theorem gives a general supercharacter formula for pattern groups.  However, to properly use the theorem we will need to choose appropriate superclass and supercharacter representatives.

\begin{theorem}[\cite{DT07}] \label{GeneralCharacterFormula} Let $u\in U_\cP$ and $\lambda\in \fkn_\cP^*$.  Then
\begin{enumerate}
\item[(a)] The character
$$\chi^\lambda(u)=0$$
unless there exists $x\in \FF_q^{|J|}$ such that $Mx=-a$ and $b\perp \Null(M)$,
\item[(b)] If $\chi^\lambda(u)$ is not zero, then
$$\chi^\lambda(u)=\frac{q^{|U_\cP \lambda|}}{q^{\rank(M)}} \theta(x\cdot b)\theta\circ\lambda(u-1),$$
where $\cdot$ is the usual inner product (dot product) on $\FF_q^{|J|}$.
\end{enumerate}
\end{theorem}

\noindent\textbf{Remark.} There are two natural choices for $\chi^\lambda$, one of which is the conjugate of the other.  Theorem \ref{GeneralCharacterFormula} uses the convention of \cite{DI06} rather than \cite{DT07}.

\vspace{.25cm}

C. Andr\'e proved the $U_n$-version of this supercharacter formula for large characteristic, and \cite{ADS04} extended it to all finite fields.  Note that the following theorem follows from Theorem \ref{GeneralCharacterFormula} by choosing appropriate representatives for the superclasses and supercharacters.
\begin{theorem}  Let $\lambda\in \cS_n(q)$, and let $u\in U_n$ be a superclass representative as in (\ref{UOrbitRepresentatives}).  Then
\begin{enumerate}
\item[(a)] The character degree 
$$\chi^\lambda(1)=\prod_{i<j, \lambda_{ij}\neq 0} q^{j-i-1}.$$
\item[(b)] The character 
$$\chi^\lambda(u)=0$$
unless whenever $u_{jk}\neq 0$ with $j<k$, we have $\lambda_{ij}=0$ for all $i<j$ and $\lambda_{jl}=0$ for all $l>k$.
\item[(c)] If $\chi^\lambda(u)\neq 0$, then
$$\chi^\lambda(u)=\frac{\chi^\lambda(1)\theta\circ\lambda(u-1)}{q^{|\{i<j<k<l\ \mid\ u_{jk},\lambda_{il}\in \FF_q^\times\}|}}.$$
\end{enumerate}
\end{theorem}

\section{Interpolating between $U_{n-1}$ and $U_n$}\label{TheInterpolatingGroups}

Fix $n\geq 1$.  For $0\leq m\leq n$, let
\begin{equation*}
\begin{split} U_{(m)} & =\{u\in U_n\ \mid\  u_{1j}=0, \text{ for $j\leq m$}\}=U_{\cP_{(m)}},\\
\fkn_{(m)} &= \{u-1\ \mid\  u\in U_{(m)}\}=\fkn_{\cP_{(m)}},\\
\fkn_{(m)}^* &=  \{\lambda:\fkn_{(m)}\rightarrow \FF_q\ \mid\  \lambda\text{ $\FF_q$-linear}\}=\fkn_{\cP_{(m)}}^*,
\end{split}
\qquad\text{where}
\qquad\cP_{(m)}=\xy<0cm,2.2cm>\xymatrix@R=.25cm@C=.3cm{*={} & *+{n}\ar @{-} [d]\\ & *{\vdots}\ar @{-} [d] \\ & *+{m+1} \ar @{-} [d] \ar @{-} [dl]\\ *+{1} & *+{m}\ar @{-} [d] \\ & *+{m-1} \ar @{-} [d]\\ & *{\vdots}\ar @{-} [d]\\ & *+{2}}\endxy.
\end{equation*}
Note that 
$$U_{n-1}\cong U_{(n)} \triangleleft U_{(n-1)} \triangleleft \cdots \triangleleft U_{(1)} \triangleleft U_{(0)} = U_n.$$
The goal of this section is to identify suitable orbit representatives for representatives for $U_{(m)}\backslash \fkn_{(m)}/ U_{(m)}$ and $U_{(m)}\backslash \fkn_{(m)}^*/ U_{(m)}$.

A matrix $A$ has an underlying vertex-labeled graph structure $G_A$ given by vertices 
$$V_A=\{\text{Nonzero entries in $A$}\}$$
and an edge from $A_{ij}$ to $A_{kl}$ if $i=k$ or $j=l$.  We label each vertex by its location in the matrix, so $A_{ij}$ has label $(i,j)$.  For example, for $a,b,c,d,e,f,g,h\in \FF_q^\times$,
$$ A=\left(\begin{array}{ccccc} 
0 & 0 & a & 0 & b\\ 
c & 0 & 0 & 0 & d\\
0 & e & 0 & f & 0\\
0 & 0 & 0 & 0 & g\\
0 & 0 & 0 & h & 0\end{array}\right) \qquad\text{implies}\qquad G_A= \left(\xy<0cm,1cm> \xymatrix@R=.25cm@C=.25cm{
{\ } & {\ } & *{a}\ar @{-} [rr] & {\ }  & *{b}\ar @{-} [d] \ar @{-} @/^{.2cm}/ [ddd]\\ 
*{c} \ar @{-} [rrrr] & {\ } & {\ } & {\ } & *{d}\ar @{-} [dd]\\
{\ } & *{e}\ar @{-} [rr] & {\ }  & *{f} \ar @{-} [dd] & \\
{\ } & {\ } & {\ } & {\ } & *{g}\\
{\ } & {\ } & {\ } & *{h} & {\ }}\endxy\right).$$

\subsection{Superclass representatives}

Unlike with $U_n$, the interpolating groups $U_{(m)}$ have several natural representatives to choose from.  In this case, we consider a ``natural choice" of an orbit representative to be one with a minimal number of nonzero entries.    This section introduces two particular examples.

A matrix $u\in U_{(m)}$ is a \emph{comb representative} if 
\begin{enumerate}
\item[(a)] At most one connected component of $G_{u-1}$ has more than one element,
\item[(b)] If $G_{u-1}$ contains a connected component $S$ with more than one element, then there exist $1\leq i_r<i_{r-1}<\cdots i_1\leq m<k_1<k_2<\cdots<k_r$ such that 
$$\hspace*{-.25cm}\left\{\begin{array}{ccccc}
 u_{1k_1} & u_{1k_2} & \cdots & u_{1k_{r-1}} & u_{1k_r}\\
 & & & u_{i_{r-1}k_{r-1}}\\
 & & \adots\\
 & u_{i_2k_2}\\
 u_{i_1k_1}\end{array}\right\}
 \quad \text{or}\quad
\left\{ \begin{array}{cccc}
 u_{1k_1} & u_{1k_2} & \cdots & u_{1k_{r}}\\
 & & & u_{i_{r}k_{r}}\\
 & & \adots\\
 & u_{i_2k_2}\\
 u_{i_1k_1}\end{array}\right\}$$
are the vertices of $S$.
\end{enumerate}

A matrix $u\in U_{(m)}$ is a \emph{path representative} if 
\begin{enumerate}
\item[(a)] At most one connected component of $G_{u-1}$ has more than one element,
\item[(b)] If $G_{u-1}$ contains a connected component $S$ with more than one element, then there exist $1<i_{r'}<i_{r'-1}<\cdots<i_1\leq m<k_1<k_2<\cdots<k_r$ with $r'\in \{r,r-1\}$ such that 
$$\hspace*{-.25cm}\left\{\begin{array}{ccccc}
 u_{1k_1} & & & \\
 & & & & u_{i_{r}k_{r}}\\
 & & & u_{i_{r-1}k_{r-1}} & u_{i_{r-1}k_{r}}\\
 & & \adots\\
 & u_{i_2k_2}\\
 u_{i_1k_1} & u_{i_1k_2}\end{array}\right\}
 \quad\text{or}\quad
\left\{ \begin{array}{ccccc}
 u_{1k_1} & & & \\
 & & & u_{i_{r-1}k_{r-1}}& u_{i_{r-1}k_{r}}\\
 & & & u_{i_{r-2}k_{r-1}} \\
 & & \adots\\
 & u_{i_2k_2}\\
 u_{i_1k_1} & u_{i_1k_2}\end{array}\right\}$$
are the vertices of $S$.
\end{enumerate}

Let
\begin{align*}
\cT_{(m)}^\vee &= \{u\in U_{(m)}\ \mid\  u\text{ a comb representative}\}\\
\cZ_{(m)}^\vee &= \{u\in U_{(m)}\ \mid\ u\text{ a path representative}\}.
\end{align*}

If $u\in \cZ_{(m)}^\vee$ has a connected component $S_u$ with a vertex in the first row, then we can order the vertices of $S_u$ by starting with the vertex in the first row and then numbering in order along the path.  For example, 
$$S_u=\left(\xy<0cm,1.5cm> \xymatrix@C=.25cm@R=.25cm{ *{x_1} \ar @{:} [ddddd] & {\ } & {\ } & {\ } &  {\ }\\ 
{\ } & {\ } & {\ } & {\ } & *{x_k}\\
{\ } & {\ } & {\ } & *{x_{k-2}}\ar @{:} [r]  & *{x_{k-1}} \ar @{:}[u]\\
{\ } & {\ } & \adots & *{x_{k-3}}\ar @{:} [u] \\
{\ } & *{x_4} & & & \\
*{x_2} \ar @{:} [r]  & *{x_3}  \ar @{:} [u] & & & }\endxy \right)
$$
where the indices of the $x$'s indicate the prescribed order.
The \emph{baggage} of $S_u$ at $x_j$ is
\begin{equation}\label{xBaggage}
\mathrm{bag}(x_j)=x_j(-x_{j-1})^{-1}x_{j-2}(-x_{j-3})^{-1}\cdots ((-1)^{j+1}x_1)^{(-1)^{j+1}}.
\end{equation}

\begin{proposition}  \label{PropositionSuperclassRepresentatives}  Let $0< m<n$.  Then 
\begin{enumerate}
\item[(a)] $\cT_{(m)}^\vee$ is a set of superclass representatives for $U_{(m)}$,
\item[(b)] $\cZ_{(m)}^\vee$ is a set of superclass representatives for $U_{(m)}$.
\end{enumerate}
\end{proposition}
\begin{proof}
(a)  First note that the above representatives are in different superclasses of $U_n$, since the corresponding representative would have $u_{i_1k_1},u_{i_2k_2},\ldots,u_{i_rk_r}$ as the only nonzero elements in their columns (by row reducing up).

Since $1$ is incomparable to $j\in \{2,3,\ldots, m\}$ in $\cP_{(m)}$, we may not add row $j$ to row $1$ if $j\leq m$ when computing superclasses.  Thus, the orbits are the same as the corresponding orbits in $U_n$ except for the first row, and every column can have at most two nonzero entries with one of those entries being in the first row.  Therefore it suffices to show that if the rows of the second nonzero entries do not decrease as we move from left to right, we can convert them into an appropriate form.  The following sequence of row and column operations effects such an adjustment.
\begin{align*} 
\left(\begin{array}{cc} u_{1k} & u_{1l}\\ u_{ik} & 0 \\  0 & u_{jl} \end{array}\right)&\overset{-u_{1k}^{-1}u_{1l}\mathrm{Col}(k)+\mathrm{Col}(l)}{\longrightarrow} \left(\begin{array}{cc} u_{1k} & 0 \\ u_{ik} & -u_{jk}u_{1k}^{-1}u_{1l}  \\  0 & u_{jl} \end{array}\right)\\
&\overset{u_{jl}^{-1}u_{jk}u_{1k}^{-1}u_{1l}\mathrm{Row}(j)+\mathrm{Row}(i)}{\longrightarrow}
\left(\begin{array}{cc} u_{1k} & 0 \\ u_{ik} & 0 \\  0 & u_{jl} \end{array}\right).
\end{align*}

(b)  Note that row and column operations imply that 
\begin{equation}\label{PathToCombClass}
\left(\begin{array}{ccccc} x_1 &  & & & \\ 
& & & & x_k\\
 & & & x_{k-2} & x_{k-1}\\
& & \adots & x_{k-3}\\
& x_4 & & & \\
x_2  & x_3 & & &  \end{array}\right)
\quad \text{and} \quad 
\left(\begin{array}{ccccc}
 x_1 & \mathrm{bag}(x_3)  & \cdots &  \mathrm{bag}(x_{k-3}) &\mathrm{bag}(x_{k-1}) \\ 
& & & & x_k\\
 & & & x_{k-2} & \\
& & \adots & \\
& x_4 & & & \\
x_2  &  & & &  \end{array}\right)
\end{equation}
are in the same superclass.  Thus, (b) follows from (a).
\end{proof}

\subsection{Supercharacter representatives}\label{SectionSuperCharacterRepresentatives}

Recall that we identify $\lambda\in \fkn_\cP^*$ with matrices $\lambda\in \fkn_\cP$ by the convention
$$\lambda_{ij}=\lambda(e_{ij}),\qquad \text{where $e_{ij}\in \fkn$ has $(i,j)$ entry 1 and zeroes elsewhere.}$$
A function $\lambda\in \fkn_{(m)}^*$ is a \emph{comb representative} if 
\begin{enumerate}
\item[(a)] At most one connected component of $G_{\lambda}$ has more than one element,
\item[(b)] If $G_\lambda$ has a connected component $S$ with more than one element, then there exist $k_1>k_2>\cdots>k_{r}>m\geq i_{r'}>i_{r'-1}>\cdots >i_1>1$ with $r'\in \{r,r-1\}$ such that
$$\hspace*{-.25cm}\left\{\begin{array}{ccccc} 
& & & & \lambda_{1k_1}\\ 
& & & \lambda_{i_1k_2} & \lambda_{i_1k_1}\\ 
& & \lambda_{i_2k_3} & & \lambda_{i_2k_1}\\
& \adots & & & \vdots\\
\lambda_{i_{r-1}k_r} & & & & \lambda_{i_{r-1}k_1}\\
& & & & \lambda_{i_rk_1}\end{array}\right\}
\quad\text{or}\quad
\left\{ \begin{array}{ccccc} 
& & & & \lambda_{1k_1}\\ 
& & & \lambda_{i_1k_2} & \lambda_{i_1k_1}\\ 
& & \lambda_{i_2k_3} & & \lambda_{i_2k_1}\\
& \adots & & & \vdots\\
\lambda_{i_{r-1}k_r} & & & & \lambda_{i_{r-1}k_1}\end{array}\right\}$$
are the vertices of $S$.
\end{enumerate}

A function $\lambda\in \fkn_{(m)}^*$ is a \emph{path representative} if 
\begin{enumerate}
\item[(a)] At most one connected component of $G_{\lambda}$ has more than one element,
\item[(b)] If $G_{\lambda}$ contains a connected component $S$ with more than one element, then there exist $k_1>k_2>\cdots>k_r>m\geq i_{r'}>i_{r'-1}>\cdots>i_1>1$ with $r'\in \{r,r-1\}$ such that 
$$\hspace*{-.5cm}\left\{\begin{array}{ccccc} 
& & & & \lambda_{1k_1}\\ 
& & & \lambda_{i_1k_2} & \lambda_{i_1k_1}\\ 
& & \lambda_{i_2k_3} & \lambda_{i_2k_2} & \\
& \adots & \adots & & \\
\lambda_{i_{r-1}k_r} & \lambda_{i_{r-1}k_{r-1}} & & & \\
\lambda_{i_rk_r} & & & & \end{array}\right\}
\ \text{or}\ 
\left\{\begin{array}{ccccc} 
& & & & \lambda_{1k_1}\\ 
& & & \lambda_{i_1k_2} & \lambda_{i_1k_1}\\ 
& & \lambda_{i_2k_3} & \lambda_{i_2k_2} & \\
& \adots & \adots & & \\
\lambda_{i_{r-1}k_r} & \lambda_{i_{r-1}k_{r-1}} & & &  \end{array}\right\}$$
are the vertices of $S$.
\end{enumerate}

Let
\begin{align*}
\cT_{(m)} &= \{\lambda\in \fkn_{(m)}^\ast\ \mid\  \lambda\text{ a comb representative}\}\\
\cZ_{(m)} &= \{\lambda\in \fkn_{(m)}^\ast\ \mid\ \lambda \text{ a path representative}\}.
\end{align*}

If $\lambda\in \cZ_{(m)}$ has a connected component $S_\lambda$ with a vertex in the first row, then we can order the vertices of $S_\lambda$ by starting with the vertex in the first row and then numbering in order along the path.   For example,
$$S_\lambda=\left(\xy<0cm,1.3cm> \xymatrix@C=.25cm@R=.25cm{ {\ }  & {\ } & {\ } & {\ } &  *{y_1} \ar @{-} [d] \\ 
{\ } & {\ } & {\ } & *{y_3}\ar @{-} [d] & *{y_2}\ar @{-} [l] \\
{\ } & {\ } & {\ } & *{y_4}  & *{\ }\\
 & & *={\adots} & \\
{\ } & *{y_{r-2}} & & & \\
*{y_r} \ar @{-} [r]  & *{y_{r-1}}  \ar @{-} [u] & & & }\endxy \right)
$$
where the indices of the $y$'s indicate the prescribed order.  The \emph{baggage} of $S_\lambda$ at $y_j$ is
\begin{equation}\label{yBaggage}
\mathrm{bag}(y_j)=y_j(-y_{j-1})^{-1}y_{j-2}(-y_{j-3})^{-1}\cdots ((-1)^{j+1}y_1)^{(-1)^{j+1}}.
\end{equation}

\begin{proposition} \label{PropositionSupercharacterRepresentatives} Let $0< m<n$.  Then 
\begin{enumerate}
\item[(a)] $\cT_{(m)}$ is a set of supercharacter representatives,
\item[(b)] $\cZ_{(m)}$ is a set of supercharacter representatives.
\end{enumerate}
\end{proposition}
\begin{proof}
(a) First note that the above representatives are in different $U_n$-orbits, since the corresponding representative would have $\lambda_{i_1k_1},\lambda_{i_2k_2},\ldots,\lambda_{i_rk_r}$ as the only nonzero elements in their rows (by row reducing down from the first row).

We may assume the first row has at most one nonzero entry by column reducing.  If it has no nonzero entry, then further reductions are the same as in $U_n$.  If the first row has a nonzero entry in column $l$, any row $j$ with $1<j\leq m$ can have at most two nonzero entries, since we can column reduce using every column except column $l$.  In fact, the only way row $j$ can have two nonzero entries is if one is in column $l$ and the other is to the left of $l$.   Note that
\begin{align*} 
\left(\begin{array}{ccc} \lambda_{ik} & 0 & \lambda_{il}  \\ 0 & \lambda_{jk'} & \lambda_{jl} \end{array}\right) &\overset{-\lambda_{il}^{-1}\lambda_{jl}\mathrm{Row}(i)+\mathrm{Row}(j)}{\longrightarrow} \left(\begin{array}{ccc} \lambda_{ik} & 0 & \lambda_{il}  \\ -\lambda_{ik'}\lambda_{il}^{-1}\lambda_{jl} & \lambda_{jk} & 0 \end{array}\right)\\
&\overset{\lambda_{jk'}^{-1}\lambda_{ik'}\lambda_{il}^{-1}\lambda_{jl} \mathrm{Col}(k')+\mathrm{Col}(k)}{\longrightarrow}
\left(\begin{array}{ccc} \lambda_{ik} & 0 & \lambda_{il}  \\ 0 & \lambda_{jk} & 0 \end{array}\right).
\end{align*}
Thus, the second nonzero entries can be arranged to occur in decreasing rows from left to right. 

It suffices to show that if a row has two nonzero entries then both columns of the nonzero entries must be to the right of the $m$th column.   Note that if $k,k'\leq m$, then
\begin{align*} 
\left(\begin{array}{ccc} \cdot & \cdot & \lambda_{1l}  \\ 0 & \lambda_{ik'} & \lambda_{il} \\
\lambda_{jk} & 0 & \lambda_{jl}\end{array}\right) &\overset{-\lambda_{jl}^{-1}\lambda_{jk}\mathrm{Col}(l)+\mathrm{Col}(k)}{\longrightarrow}\left(\begin{array}{ccc} \cdot & \cdot &\lambda_{1l}\\ -\lambda_{il}\lambda_{jl}^{-1}\lambda_{jk}  & \lambda_{ik'} & \lambda_{il} \\
0 & 0 & \lambda_{jl}\end{array}\right)\\
&\overset{\lambda_{ik'}^{-1}\lambda_{il}\lambda_{jl}^{-1}\lambda_{jk} \mathrm{Col}(k')+\mathrm{Col}(k)}{\longrightarrow}
\left(\begin{array}{ccc} \cdot & \cdot &\lambda_{1l}\\ 0  & \lambda_{ik'} & \lambda_{il} \\
0 & 0 & \lambda_{jl}\end{array}\right)\\
&\overset{-\lambda_{il}^{-1}\lambda_{jl}\mathrm{Row}(i)+\mathrm{Row}(j)}{\longrightarrow}
\left(\begin{array}{ccc} \cdot & \cdot &\lambda_{1l}\\ 0  & \lambda_{ik'} & \lambda_{il} \\
0 & -\lambda_{ik'}\lambda_{il}^{-1}\lambda_{jk} & 0\end{array}\right)\\
&\overset{-\lambda_{il}^{-1}\lambda_{ik'}\mathrm{Col}(l)+\mathrm{Col}(k')}{\longrightarrow}
\left(\begin{array}{ccc} \cdot & \cdot &\lambda_{1l}\\ 0  & 0 & \lambda_{il} \\
0 & -\lambda_{ik'}\lambda_{il}^{-1}\lambda_{jk} & 0\end{array}\right)
\end{align*}
Thus, any nonzero pair of nonzero entries in a row must occur to the right of column $m$.

(b)  Note that row and column operations imply that 
\begin{equation}\label{PathToCombCharacter}
\left(\begin{array}{ccccc}  &  & & & y_1\\ 
& & & y_3 & y_2\\
 & & \adots & y_4 & \\
& y_{k-3} &  & \\
y_{k-1} &  y_{k-2} & & & \\
y_k & & & &  \end{array}\right)
\quad\begin{array}{c} \text{and}\end{array}  \quad
\left(\begin{array}{ccccc}  &  & & & y_1\\ 
& & & y_3 & -y_1\mathrm{bag}(y_2)\\
 & & \adots &  & \vdots \\
& y_{k-3} &  & & -y_1\mathrm{bag}(y_{k-4})\\
y_{k-1} &   & & & -y_1\mathrm{bag}(y_{k-2})\\
 & & & &  -y_1\mathrm{bag}(y_{k}) \end{array}\right)
 \end{equation}
 are give rise to the same supercharacter, so (b) follows from (a).
\end{proof}

\section{Supercharacter formulas for $U_{(m)}$}

This section develops supercharacter formulas for both comb and path representatives.  After developing tools that allow us to decompose characters as products of simpler characters, we prove a character formula for comb characters.  We then use the translation between comb and path representatives of (\ref{PathToCombClass}) and (\ref{PathToCombCharacter}) to get a more combinatorial character formula for path representatives.

\subsection{Multiplicativity of supercharacter formulas}

Let $u\in U_{(m)}$.  For a connected component $S$ of $G_{u-1}$, let $u[S]\in U_{(m)}$ be given by
$$u[S]_{jk}=\left\{\begin{array}{ll} u_{jk}, & \text{if $u_{jk}\in V_S$,}\\ 0, & \text{otherwise.}\end{array}\right.$$

Similarly, let $\lambda\in \fkn_{(m)}^\ast$.  For a connected component $T$ of $G_\lambda$, let $\lambda[T]\in \fkn_{(m)}^*$ be given by
$$\lambda[T]_{jk}=\left\{\begin{array}{ll} \lambda_{jk}, & \text{if $\lambda_{jk}\in V_T$,}\\ 0, & \text{otherwise.}\end{array}\right.$$

The following lemma allows us to decompose the supercharacter formula of a pattern group $U_\cP$ by connected components. 

\begin{lemma}\label{DecomposingCharacterConjugacyComponents}
Let $u\in U_\cP$ and $\lambda\in \fkn_\cP^\ast$.  Let $S_1,S_2,\ldots, S_k$ be the connected components of $G_{u-1}$ and $T_1,T_2,\ldots, T_l$ be the connected components of $G_\lambda$. Then
$$\chi^\lambda(u)= \prod_{j=1}^l  \chi^{\lambda[T_j]}(1) \prod_{i=1}^k \frac{\chi^{\lambda[T_j]}(u[S_i])}{\chi^{\lambda[T_j]}(1)}.$$
\end{lemma}
\begin{proof}
Let $U=U_\cP$.  The proof follows from the following two claims:
\begin{enumerate}
\item[(1)]  If $\lambda$ has two connected components $S$ and $T$, then
$$\chi^\lambda(u)=\chi^{\lambda[S]}(u)\chi^{\lambda[T]}(u).$$
\item[(2)] If $u$ has two connected components $S$ and $S'$, then
$$\chi^\lambda(u)=\chi^\lambda(1) \frac{\chi^\lambda(u[S])}{\chi^\lambda(1)}\frac{\chi^\lambda(u[S'])}{\chi^\lambda(1)}.$$
\end{enumerate}
(1) Note that since $T$ and $T'$ involve distinct rows and columns, the left orbits of $\lambda[T]$ and $\lambda[T']$ are independent and involve distinct rows.  Thus,
$$|U\lambda|=|U \lambda[T]||U\lambda[T']|.$$
In fact, for $\lambda'\in U\lambda U$, 
$$|\{(\gamma,\mu)\in (U\lambda[T] U)\times (U\lambda[T']U)\ \mid\ \lambda'=\gamma+\mu\}|=\frac{|U \lambda[T]U||U\lambda[T']U|}{|U\lambda U|}.$$
Thus, by definition
\begin{align*}
\chi^\lambda(u) & =\frac{|U\lambda|}{|U\lambda U|}\sum_{\lambda'\in U\lambda U} \theta(-\lambda'(u-1))\\
&  =\frac{|U\lambda|}{|U\lambda[T] U||U\lambda[T'] U|}\sum_{\gamma\in U\lambda[T] U\atop \mu\in U\lambda[T']U} \theta\bigl(-\gamma(u-1)-\mu(u-1)\bigr)\\
&  =\frac{|U\lambda[T]||U\lambda[T']|}{|U\lambda[T] U||U\lambda[T'] U|}\sum_{\gamma\in U\lambda[T] U\atop \mu\in U\lambda[T']U} \theta\bigl(-\gamma(u-1)\bigr)\theta\bigl(-\mu(u-1)\bigr)\\
&  =\frac{|U\lambda[T]|}{|U\lambda[T] U|}\sum_{\gamma\in U\lambda[T] U}\theta\bigl(-\gamma(u-1)\bigr) \frac{|U\lambda[T']|}{|U\lambda[T'] U|}\sum_{\mu\in U\lambda[T']U} \theta\bigl(-\mu(u-1)\bigr)\\
&= \chi^{\lambda[T]}(u)\chi^{\lambda[T']}(u).
\end{align*}

(2)  For any $u'\in UuU$,
$$|\{(v,w)\in (Uu[S] U)\times (Uu[S']U)\ \mid\ u'-1=v-1+w-1\}|=\frac{|U u[S]U||Uu[S']U|}{|UuU|}.$$
We have that
\begin{align*}
\chi^\lambda(u)&=\frac{\chi^\lambda(1)}{|U u U|}\sum_{v\in U u U} \theta(-\lambda(v-1))\\
&=\frac{\chi^\lambda(1)}{|U u[S] U||U u[S'] U}\sum_{v\in U u[S] U\atop w\in Uu[S'] U} \theta(-\lambda(v-1+w-1))\\
&=\frac{\chi^\lambda(1)}{\chi^\lambda(1)\chi^\lambda(1)} \frac{\chi^\lambda(1)}{|U u[S] U|}\sum_{v\in Uu[S]U}\theta(-\lambda(v-1))\frac{\chi^\lambda(1)}{|U u[S'] U|}\sum_{w\in Uu[S']U}\theta(-\lambda(w-1))\\
&=\chi^\lambda(1) \frac{\chi^\lambda(u[S])}{\chi^\lambda(1)}\frac{\chi^\lambda(u[S'])}{\chi^\lambda(1)},
\end{align*}
as desired.
\end{proof}

\begin{corollary} \label{CharacterMultiplicative} Let $u\in U_\cP$ and $\lambda\in \fkn_\cP^*$ with connected components $T_1, \ldots, T_l$.  Then
$$\chi^\lambda(u)=\prod_{i=1}^l \chi^{\lambda[T_i]}(u),$$
\end{corollary}

To obtain character formulas for $U_{(m)}$ we will require a slightly more refined multiplicativity result that depends on the poset structure $\cP_{(m)}$ and a choice of comb representatives.

For $u\in U_{(m)}$ and $1\leq k\leq n$, let $u[k]\in U_{(m)}$ be given by
$$u[k]_{ij}=\left\{\begin{array}{ll} u_{ij}, & \text{if $j=k$},\\ 0, & \text{otherwise.}\end{array}\right.$$
That is, $u[k]$ is the same as $u$ in the $k$th column, but zero elsewhere.
For $\lambda\in \fkn_\cP^*$, let $\lambda[u,k]\in \fkn_\cP^*$ be given by
$$\lambda[u,k]_{il}=\left\{\begin{array}{ll} \lambda_{il}, & \text{if $l\geq k$ and $u_{jk}\neq 0$ for some $j\geq i$,}\\ 0, & \text{otherwise.}\end{array}\right.$$
That is, $\lambda[u,k]$ is the same as $\lambda$ weakly NorthEast of the nonzero entries of $u$ in the $k$th column, but has zeroes elsewhere.

The following lemma states that we can compute supercharacter formulas for $U_{(m)}$ column by column on the superclasses.

\begin{lemma}\label{ClassMultiplicative}
Let $u\in U_{(m)}$ with $u\in \cT_{(m)}^\vee$ and let $\lambda\in \cT_{(m)}$.  Then
\begin{enumerate}
\item[(a)] The character $\chi^\lambda(u)\neq 0$ if and only if $\chi^{\lambda[u,k]}(u[k])\neq 0$ for all $2\leq k\leq n$.
\item[(b)] The character value
$$\chi^\lambda(u)=\chi^\lambda(1)\prod_{k=2}^n\frac{\chi^{\lambda[u,k]}(u[k])}{\chi^{\lambda[u,k]}(1)}.$$
\end{enumerate}
\end{lemma}
\begin{proof}
(a)  Let $M$ correspond to $(\lambda,u)$ as in Theorem \ref{GeneralCharacterFormula}.  Note that $M_{(i,j),(k.l)},M_{(i,j),(k',l')}\in \FF_q^\times$ implies $\lambda_{il},u_{jk},\lambda_{il'},u_{jk'}\in \FF_q^\times$, so 
$$u=\begin{array}{@{}r@{} c} & \begin{array}{ccc} & \ss{k} & \ss{k'}\end{array}\\ \begin{array}{r@{}} {}\\ \ss{j}\\ {}\end{array} & \left(\begin{array}{ccc} & & \\ u_{jk} & u_{jk'} & \\ & &  \end{array}\right)\end{array}\qquad \text{and} \qquad \lambda=\begin{array}{@{}r@{} c} & \begin{array}{ccc} & \ss{l} & \ss{l'}\end{array}\\ \begin{array}{r@{}} \ss{i}\\ {}\\ {}\end{array} & \left(\begin{array}{ccc} & \lambda_{il} & \lambda_{il'}\\ & &  \\
 & & \end{array}\right)\end{array}.$$
 However, since $u\in \cT_{(m)}^\vee$, the only row of $u$ which can have more than one nonzero entry is row 1.  Since $i<j$, we have $k=k'$ and the nonzero entries of $u$ contribute to distinct rows of $M$.  Similiarly, if $M_{(i,j),(k.l)},M_{(i',j'),(k,l)}\in \FF_q^\times$ implies $\lambda_{il},u_{jk},\lambda_{i'l},u_{j'k}\in \FF_q^\times$, so 
$$u=\begin{array}{@{}r@{} c} & \begin{array}{ccc} & \ss{k} & {}\end{array}\\ \begin{array}{r@{}} {}\\ \ss{j}\\ \ss{j'}\end{array} & \left(\begin{array}{ccc} & & \\ & u_{jk}   & \\   & u_{j'k} &  \end{array}\right)\end{array}\qquad \text{and} \qquad \lambda=\begin{array}{@{}r@{} c} & \begin{array}{ccc} &  & \ss{l}\end{array}\\ \begin{array}{r@{}} \ss{i}\\ \ss{i'}\\ {}\end{array} & \left(\begin{array}{ccc} & & \lambda_{il}\\ & & \lambda_{i'l} \\
 & & \end{array}\right)\end{array}.$$
Thus, distinct columns of $u$ contribute to distinct columns of $M$. For $1\leq k\leq n$,
\begin{equation}\label{RandCDefinition}
\begin{split}
R_k & = \begin{array}{l}\text{rows of $M$ that have nonzero entries corresponding}\\ \text{to the nonzero entries of $u$ in column $k$}\end{array}\\
C_k&=\begin{array}{l}\text{columns of $M$ that have nonzero entries corresponding}\\ \text{to the nonzero entries of $u$ in column $k$}\end{array}
\end{split}
\end{equation}
By choosing an appropriate order on the rows and columns of $M$, 
\begin{equation} \label{MDecomposition}
M=M_{R_1,C_1}\oplus M_{R_2,C_2}\oplus\cdots \oplus M_{R_n,C_n},
\end{equation}
where $M_{R_k,C_k}$ is the submatrix of $M$ using rows $R_k$ and columns $C_k$.

Using (\ref{MDecomposition}), there exists a solution to $Mx=-a$ if and only if for each $1\leq k\leq n$, there exist $x_k\in \FF_q^{C_k}$ such that $M_{R_k,C_k} x_k=-a_{R_k}$. 

If $a_{ij}\neq 0$, then there exist $\lambda_{ik},u_{jk}\in \FF_q^\times$ for some 
$k$.  Since row $j$ in $u$ has at most one nonzero entry, $a_{ij}=u_{jk}\lambda_{ik}$ .  Thus, $a_{R_k}$ only depends on the pair $(\lambda[u,k],u[k])$.

By (\ref{MDecomposition}), we have
$$\Null(M)= \Null(M_{R_1,C_1})\oplus \Null(M_{R_2,C_2})\oplus\cdots\oplus \Null(M_{R_n,C_n}),$$
so $b$ is perpendicular to $\Null(M)$ if and only if $b_{C_k}$ is perpendicular to $M_{R_k,C_k}$ for all $k$.   The condition $(k,l)\in C_k$ implies $u_{jk}\neq 0$ for some $j$, so $b_{kl}\in \FF_q^\times$ implies $b_{kl}=u_{1k}\lambda_{1l}+u_{jk}\lambda_{jl}$.  Thus, $b_{C_k}$ only depends on the pair $(\lambda[u,k],u[k])$, and (a) follows.

(b) Since $C_1=R_1=\emptyset$, it follows from (\ref{MDecomposition}) that
$$\rank(M)=\sum_{k=1}^n \rank(M_{R_k,C_k})=\sum_{k=2}^n\rank(M_{R_k,C_k}).$$
It follows from (a) that
$$\theta(x\cdot b)=\prod_{k=2}^n \theta(x_{C_k}\cdot b_{C_k}),$$
and by inspection
$$\theta\circ\lambda(u-1)=\prod_{(j,k)} \theta(u_{jk}\lambda_{jk})=\prod_{k=1}^n \prod_{j<k} \theta(u_{jk}\lambda_{jk})=\prod_{k=1}^n \theta\bigl(\lambda[u,k](u[k]-1)\bigr).$$
Now (b) follows from (a).
\end{proof}

\noindent\textbf{Remark.} This lemma depends on the choice of representatives.  In particular, it is not true for path representatives.

\subsection{A character formula for comb representatives}

It follows from Lemmas \ref{CharacterMultiplicative} and \ref{ClassMultiplicative} that to give the character value of $\chi^\lambda(u)$, we may assume $u-1$ has nonzero entries in one column and $G_\lambda$ has one connected component $S$.

\begin{theorem}  \label{TriangleCharacterFormula} Let $u\in U_{(m)}$ such that $u\in \cT_{(m)}^\vee$ and $u-1$ has support $\supp(u-1)\subseteq\{(1,k),(j,k)\}$.  Let $\lambda\in \cT_{(m)}$ be such that $\lambda$ has one connected component $S$ with $\Cols(S)=\{l_1<l_2<\cdots< l_s\}$.  Then

 \vspace{.25cm}
 
 \noindent (a) Let $i_1>i_2>\ldots>i_{s-1}$ be such that $\lambda_{i_dl_d}\neq 0$.  The character degree
 $$\chi^\lambda(1)=\left\{\begin{array}{ll}\dd q^{l_s-m-2}\prod_{d=1}^{s-1} q^{l_d-i_d-1}, & \text{if $\lambda_{il_s}=0$ for all $i>i_1$},\\ \dd q^{l_s-i-1}\prod_{d=1}^{s-1} q^{l_d-i_d-1}, & \text{if $\lambda_{il_s}\neq 0$ for some $i>i_1$.}\end{array}\right.$$

\vspace{.25cm}

\noindent (b)  The character 
$$\chi^\lambda(u)=0$$
unless at least one of the following occurs
\begin{enumerate}
\item[(1)]   $u_{jk}\lambda_{ik}\neq 0$ implies $i=1$ with $j\leq m$ or $i>j$; and $u_{1k}\lambda_{1l}+u_{jk}\lambda_{jl}=0$ for all $j<k<l$,
\item[(2)] $u_{jk}\lambda_{ik}\neq 0$ for some $1<i<j\leq m$, but $|R_k|=|C_k|>0$ ($R_k$ and $C_k$ are as in (\ref{RandCDefinition})), 
\item[(3)] $u_{1k}\lambda_{1l_s}+u_{jk}\lambda_{jl}\neq 0$ for some $m<k<l$, but $\lambda_{ik}=0$ for all $i$ and $|R_k|\geq |C_k|>0$,
\item[(4)] $u_{jk},\lambda_{jl_s},\lambda_{il_s}\in \FF_q^\times$ with $i<j'<k<l_s$ with $\lambda_{jk'}=0$ for all $k<k'<l_s$.
\end{enumerate}

\vspace{.25cm}

\noindent (c)  The character values are
$$\chi^\lambda(u)=\left\{\begin{array}{ll} \frac{\chi^\lambda(1)}{q^{|C_k|-\delta_{RC}}}\theta(u_{jk}\lambda_{jk}), & \text{if (1)}\\
\frac{\chi^\lambda(1)}{q^{|C_k|}}, & \text{if (2) or (3) or (4)}\\ 
\frac{\chi^\lambda(1)}{q^{|C_k|}}\theta\bigl(-\lambda_{il_s}^{-1}\lambda_{ik}(u_{1k}\lambda_{1l_s}+u_{jk}\lambda_{jl_s})\big), & \text{if (2) and (4),}\end{array}\right.$$
where $\delta_{RC}=1$ if $|C_k|>|R_k|$ and $\delta_{RC}=0$ if $|C_k|\leq |R_k|$.
\end{theorem}

\begin{proof}  (a) This is just a statement of the fact that
$$\chi^\lambda(1)=|U_{(m)} \lambda|,$$
combined with the structure of $S$.

(b) and (c). Note that by Lemma \ref{ClassMultiplicative},
$$\chi^{\lambda}(u)=\frac{\chi^\lambda(1)}{\chi^{\lambda[u,k]}(1)}\chi^{\lambda[u,k]}(u[k]),$$ 
so we may assume $M=M_{R_k,C_k}$ (see (\ref{MDecomposition})).  Let $\Rows(S)=\{i_1>\ldots> i_{s}\}$ or $\Rows(S)=\{i_0>i_1>\ldots>i_s\}$ be the rows with nonzero entries in $S$ such that $\lambda_{i_d l_d},\lambda_{i_d l_s}\in \FF_q^\times$, and, if $i_0\in \Rows(S)$, then $\lambda_{i_0l_s}\neq 0$ (see the definition of comb representatives in Section \ref{SectionSuperCharacterRepresentatives}).  Let $r$ and $r'$ be minimal such that $l_r>k$ and $i_{r'}<j$.  Then
\begin{equation}\label{SpecificM} 
M=\left(\begin{array}{ccccc}
&  & & & \delta'u_{jk}\lambda_{1l_s}\\ 
& & & u_{jk}\lambda_{i_{s-1}l_{s-1}} & u_{jk}\lambda_{i_{s-1}l_s}\\
 & & \adots & & \vdots \\
 & u_{jk}\lambda_{i_rl_r} & & &  u_{jk}\lambda_{i_rl_s}\\
& & & & u_{jk}\lambda_{i_{r-1}l_s}\\
& & & & \vdots\\
& & & & u_{jk} \lambda_{i_{r'}l_s}\end{array}\right), \qquad\text{where}\quad \delta'=\left\{\begin{array}{ll} 1, & \text{if $j>m$,}\\ 0, & \text{if $j\leq m$}\end{array}\right.\end{equation}
Thus, the rank of $M$ is $q^{|C_k|-\delta}$.

Furthermore, $a\in \FF_q^{|R_k|}$ and $b\in \FF_q^{|C_k|}$ are given by
\begin{align*}
a_{ij}&= u_{jk}\lambda_{ik}, & & \text{for $(i,j)\in R_k$,}\\
 b_{kl}&=\left\{\begin{array}{ll} u_{1k}\lambda_{1l}+u_{jk}\lambda_{jl}, & \text{if $l\in \Cols(S)$,}\\ 0 & \text{otherwise.}\end{array}\right. && \text{for $(k,l)\in C_k$.}
\end{align*}
If $a=0$ then $M\cdot 0=-0$ is easily satisfied, and if $b=0$ then $b\perp \Null(M)$ is also trivially satisfied.  Thus, $\chi^\lambda(u)\neq 0$ if $u_{jk}\lambda_{ik}=0$ for all $i<j<k$ in $\cP_{(m)}$ and $u_{1k}\lambda_{1l}+u_{jk}\lambda_{jl}=0$ for all $1<j<k<l$.  Note that in the poset $\cP_{(m)}$, $1\nleq j$ for $j\leq m$.

Suppose $a_{ij}\neq 0$.  Note that $Mx=-a$ can only be satisfied if row $(i,j)$ of $M$ has a nonzero element.  That is, there exists $i<j<k<l$ such that $\lambda_{il}\neq 0$.  Consequently, we may assume  $k<l_s$.  If $j>m$, then $\delta=1$, so $(Mx)_{1j}\neq 0$ if and only if $(Mx)_{ij}\neq 0$.  However, $a_{1j}\neq 0$ and $(Mx)_{1j}\neq 0$ implies the first row of $\lambda$ has two nonzero elements, contradicting the structure of $S$.  Thus, if $a_{ij}\neq 0$ and $Mx=-a$ for some $x$, then $j\leq m$ and $k<l_s$.  

Suppose $a_{ij}=u_{jk}\lambda_{ik}\neq 0$ with $j\leq m$ and $k<l_s$.  By (\ref{SpecificM}), $(i,k)=(i_{r-1},l_{r-1})$.  Note that $(Mx)_{ij}\neq 0$ if and only if $(Mx)_{i_{r'}j}\neq 0$.  Since $u_{jk'}=0$ for all $k'\neq k$, in this case $r'=r-1$ or $|C_k|=|R_k|$.   If we choose $x$ such that 
$$x_{kl}=\left\{\begin{array}{ll} -\lambda_{il_s}^{-1}\lambda_{ik}, & \text{if $l=l_s$,}\\ \lambda_{i_dl_d}^{-1}\lambda_{i_dl_s}\lambda_{il_s}^{-1}\lambda_{ik} , & \text{if $l=l_d$,}\\ 0,  & \text{otherwise,}\end{array}\right. \qquad\text{where $(k,l)\in C_k$,}$$
then $Mx=-a$.

If $b_{kl}\neq 0$ and $M$ has no nonzero entry in column $(k,l)$, then $b$ is not perpendicular to $\Null(M)$.  Thus, if $b_{kl}\neq 0$ we must have $\lambda_{jl},\lambda_{il}\in \FF_q^\times$ with $i<j$.  In particular, $j\leq m$, and either $u_{1k}=0$ or $u$ has two nonzero elements.  Since only the last column of $S$ can have more than one nonzero entry, $l=l_s$, and $b_{kl_s}=u_{1k}\lambda_{1l_s}+u_{jk}\lambda_{jl_s}$.  Note that
$$\dim(\Null(M))=\left\{\begin{array}{ll} s-r, & \text{if $\delta'=0$, $r'=r$,}\\ 0, & \text{otherwise}.\end{array}\right.$$
It follows that when $b\neq 0$, then $b$ is perpendicular to $\Null(M)$ if and only if $r'>r$ if and only if $|R_k|\geq |C_k|$ (if $\delta'=1$, then $j>m$).  

In the case that $j=i_{r-2}$ and $k=l_{r-1}$, we have
$$\theta(x\cdot b)=\theta\bigl(-\lambda_{il_s}^{-1}\lambda_{ik}(u_{1k}\lambda_{1l_s}+u_{jk}\lambda_{jl_s})\bigr),\qquad \text{where $i=i_{r-1}$.}$$
Otherwise, $\theta(x\cdot b)=1$.  
\end{proof}

At this point, it may be helpful to give a more visual interpretation of the conditions in Theorem \ref{TriangleCharacterFormula} by considering the configurations of superimposed graphs $G_\lambda$ and $G_u$.  Recall, for $\lambda\in \cZ_{(m)}\cup\cT_{(m)}$ there is at most one connected component of $G_\lambda$ that can have more than one element (or can have a vertex in the first row of $\lambda$).   Therefore, for $\lambda\in \cZ_{(m)}\cup\cT_{(m)}$, let 
\begin{align*}
S_\lambda&=\text{the connected component of $G_\lambda$ that has a vertex in the first row}\\
\cend(\lambda)&=\left\{\begin{array}{ll} \dd\min\{k\ \mid\ \text{$S_\lambda$ has a vertex in column $k$}\}, 
& \text{if $S_\lambda\neq\emptyset$,}\\ 0, & \text{otherwise,}\end{array}\right.\\
\br(\lambda)&=\left\{\begin{array}{ll} \dd\max\{j\ \mid\ \text{$S_\lambda$ has a vertex in row $j$}\}, 
& \text{if $S_\lambda\neq\emptyset$,}\\ n, & \text{otherwise,}\end{array}\right.\\
\wt(\lambda)&=\left\{\begin{array}{ll}\#(\text{Nonzero entries in row $\br(\lambda)$ of $\lambda$})-1, & \text{if $S_\lambda\neq \emptyset$},\\ 0, & \text{otherwise.}\end{array}\right.
\end{align*}
For example, if
$$\lambda=
\left(\begin{array}{cccccc}  0 & 0 & 0 & 0 & 0 & a\\
				            0 & 0 & 0 & 0 & c & b\\
				            0 & 0 & 0 & e & 0 & 0\\
				            0 & 0 & 0 & 0 & 0 & d\\
				            0 & 0 & 0 & 0 & 0 & 0\\
				            0 & 0 & 0 & 0 & 0 & 0\end{array}\right),\quad\text{then}\quad
S_\lambda=\xy<0cm,.9cm>\xymatrix@R=.35cm@C=.35cm{ *={} & *{a} \ar @{-} [d]  \\
 *{c}   & *{b} \ar @{-} [l] \ar @{-} [dd]\\
 *={} \\
 *={} & *{d}}\endxy,\quad \begin{array}{l} \cend(\lambda)=5,\\ \br(\lambda)=4,\\ \wt(\lambda)=0.\end{array}$$

In the following discusion, we will suppress the values of the vertices and distinguish between $G_{u}$ and $G_\lambda$ by the following conventions,
$$\begin{array}{|l|c|c|} \hline & G_{u-1} & G_\lambda  \\  \hline
 \text{Vertices} & \bullet & \sss\blacksquare\\  \hline
 \text{Edges} & \xymatrix@C=.5cm@R=.2cm{*={\bullet}\ar @{:} [r] & *={\bullet}} & \xymatrix@C=.5cm@R=.2cm{*={\sss\blacksquare}\ar @{-} [r] & *={ \sss\blacksquare} } \\
 \hline\end{array}$$
If $|S_\lambda|>1$ and $\wt(\lambda)=0$, then add an edge to the non-zero vertex of row $\br(\lambda)$ that extends West of this vertex,
$$
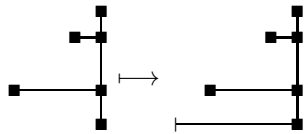
\xy<0cm,.9cm>\xymatrix@R=.35cm@C=.35cm{*={} & *={} & *={} & *={} & *={\sss\blacksquare} \ar @{-} [dddd]  \\
*={} & *={} & *={} & *={\sss\blacksquare}   & *={\sss\blacksquare} \ar @{-} [l] \\
*={} & *={} & *={} & *={} \\
*={}   & *={\sss\blacksquare} & *={}  &*={} & *={\sss\blacksquare} \ar @{-} [lll] \\
*={} & & *={}& *={} & *={\sss\blacksquare}  \\  *{}}\endxy\ \longmapsto\ 
\xy<0cm,.9cm>\xymatrix@R=.35cm@C=.35cm{*={} & *={} & *={} & *={} & *={\sss\blacksquare} \ar @{-} [dddd]  \\
*={} & *={} & *={} & *={\sss\blacksquare}   & *={\sss\blacksquare} \ar @{-} [l] \\
*={} & *={} & *={} & *={} \\
*={}   & *={\sss\blacksquare} & *={}  &*={} & *={\sss\blacksquare} \ar @{-} [lll] \\
*={} & & *={}& *={} & *={\sss\blacksquare} \ar @{-|} [llll] \\  *{}}\endxy\ ,$$
thereby ``completing" the comb.   

Vertices of $G_u$ \emph{see} North in their column and East in their row, while vertices of $G_\lambda$ \emph{see} South in their column and West in their row (in both cases they do not see the location they are in).  That is,
$$
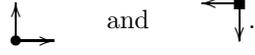
\xy<0cm,.25cm>\xymatrix@R=.5cm@C=.5cm{*={} & *={}\\ *={\bullet} \ar [r] \ar [u]  & *={}}\endxy \qquad\text{and}\qquad 
\xy<0cm,.25cm>\xymatrix@R=.5cm@C=.5cm{*={} & *={\sss\blacksquare}\ar [d] \ar [l]\\ *={} & *={}}\endxy\ .$$
Connected component $S$ of $G_u$ and $T$ of $G_\lambda$ \emph{see one-another} if when one superimposes their matrices, a vertex of $S$ sees a vertex of $T$ (and vice-versa).  

The \emph{tines} of $S_\lambda$ are the pairs of horizontal edges with their leftmost vertices.  For example, the tines of 
$$
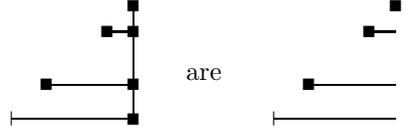
\xy<0cm,.9cm>\xymatrix@R=.35cm@C=.35cm{*={} & *={} & *={} & *={} & *={\sss\blacksquare} \ar @{-} [dddd]  \\
*={} & *={} & *={} & *={\sss\blacksquare}   & *={\sss\blacksquare} \ar @{-} [l] \\
*={} & *={} & *={} & *={} \\
*={}   & *={\sss\blacksquare} & *={}  &*={} & *={\sss\blacksquare} \ar @{-} [lll] \\
*={} & & *={}& *={} & *={\sss\blacksquare} \ar @{-|} [llll] \\  *{}}\endxy \qquad\text{are}\qquad 
\xy<0cm,.9cm>\xymatrix@R=.35cm@C=.35cm{*={} & *={} & *={} & *={} & *={\sss\blacksquare}  \\
*={} & *={} & *={} & *={\sss\blacksquare}   & *={} \ar @{-} [l] \\
*={} & *={} & *={} & *={} \\
*={}   & *={\sss\blacksquare} & *={}  &*={} & *={} \ar @{-} [lll] \\
*={} & & *={}& *={} & *={} \ar @{-|} [llll] \\  *{}}\endxy\ .$$

Suppose $u\in U_{(m)}$ has at most two nonzero superdiagonal entries $u_{1k}, u_{jk}\in \FF_q$, for some $1\leq k\leq n$, and suppose $\lambda\in \fkn_{(m)}^*$ such that $G_\lambda$ has exactly one connected component $S$.  Then column $k$ of $u$ is \emph{comb compatible} with $S$ if the following conditions are satisfied.
\begin{enumerate}
\item[(CC1)] If column $k$ of $u$ has exactly one nonzero entry $u_{jk}$ in column $k$ and $S\neq S_\lambda$, then $u_{jk}$ cannot see $S$,
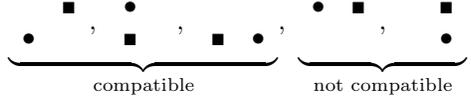
$$\underbrace{\begin{array}{cc} & \sss\blacksquare\\ \bullet & \end{array},\ 
\begin{array}{cc} \bullet & \\ \sss\blacksquare & \end{array},\ 
\begin{array}{cc} & \\ \sss\blacksquare & \bullet \end{array}}_{\text{compatible}},\ \underbrace{
\begin{array}{cc}\bullet &\sss\blacksquare \\  &  \end{array},\ 
\begin{array}{cc} & \sss\blacksquare  \\& \bullet \end{array}
}_{\text{not compatible}}.$$
\item[(CC2)]  If $u_{1k},u_{jk}\in \FF_q^\times$, with $1<j$ and $S\neq S_\lambda$, then $S$ cannot see $u_{ik}$ or $u_{jk}$,
 $$\underbrace{\begin{array}{cc} \bullet & \\
 & \sss\blacksquare\\ \bullet & \end{array},\ 
\begin{array}{cc} \bullet & \\ \bullet & \\ \sss\blacksquare & \end{array},\ 
\begin{array}{cc} & \bullet  \\ \sss\blacksquare & \bullet \end{array}}_{\text{compatible}},\ \underbrace{
\begin{array}{cc}\bullet &\sss\blacksquare \\ \bullet  &  \end{array},\ 
\begin{array}{cc} & \bullet  \\ & \sss\blacksquare  \\& \bullet \end{array}
}_{\text{not compatible}}.$$
\item[(CC3)]  If column $k$ of $u$ has exactly one nonzero entry $u_{jk}$ in column $k$ and $S=S_\lambda$, then $u_{jk}$ sees $S$ if and only if $j\leq m$, $u_{jk}$ is South of the end of a tine and weakly North of the next tine to the South (if there is another tine),
$$\underbrace
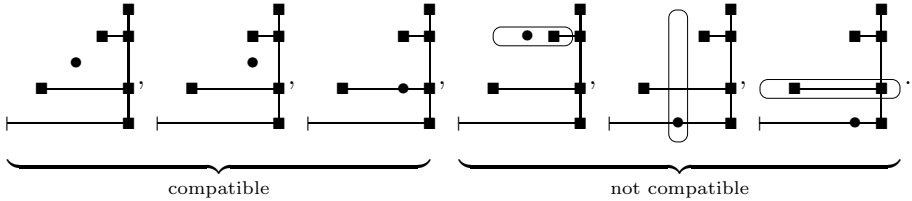
{
\xy<0cm,.9cm>\xymatrix@R=.35cm@C=.35cm{*={} & *={} & *={} & *={} & *={\sss\blacksquare} \ar @{-} [dddd]  \\
*={} & *={} & *={} & *={\sss\blacksquare}   & *={\sss\blacksquare} \ar @{-} [l] \\
*={} & *={} & *={\bullet} & *={} \\
*={}   & *={\sss\blacksquare} & *={}  &*={} & *={\sss\blacksquare} \ar @{-} [lll] \\
*={} & & *={}& *={} & *={\sss\blacksquare} \ar @{-|} [llll] \\  *{}}\endxy\ ,\ 
\xy<0cm,.9cm>\xymatrix@R=.35cm@C=.35cm{*={} & *={} & *={} & *={} & *={\sss\blacksquare} \ar @{-} [dddd]  \\
*={} & *={} & *={} & *={\sss\blacksquare}   & *={\sss\blacksquare} \ar @{-} [l] \\
*={} & *={} & *={} & *={\bullet} \\
*={}   & *={\sss\blacksquare} & *={}  &*={} & *={\sss\blacksquare} \ar @{-} [lll] \\
*={} & & *={}& *={} & *={\sss\blacksquare} \ar @{-|} [llll] \\  *{}}\endxy\ ,\ 
\xy<0cm,.9cm>\xymatrix@R=.35cm@C=.35cm{*={} & *={} & *={} & *={} & *={\sss\blacksquare} \ar @{-} [dddd]  \\
*={} & *={} & *={} & *={\sss\blacksquare}   & *={\sss\blacksquare} \ar @{-} [l] \\
*={} & *={} & *={} & *={} \\
*={}   & *={\sss\blacksquare} & *={}  &*={\bullet} & *={\sss\blacksquare} \ar @{-} [lll] \\
*={} & & *={}& *={} & *={\sss\blacksquare} \ar @{-|} [llll] \\  *{}}\endxy}_{\text{compatible}}\ ,\ 
\underbrace{\xy<0cm,.9cm>\xymatrix@R=.35cm@C=.35cm{*={} & *={} & *={} & *={} & *={\sss\blacksquare} \ar @{-} [dddd]  \\
*={} & *={} & *={\bullet} & *={\sss\blacksquare}   & *={\sss\blacksquare} \ar @{-} [l] \\
*={} & *={} & *={} & *={} \\
*={}   & *={\sss\blacksquare} & *={}  &*={} & *={\sss\blacksquare} \ar @{-} [lll] \\
*={} & & *={}& *={} & *={\sss\blacksquare} \ar @{-|} [llll] \\  *{}\save "2,2"."2,4"*+<.25cm>[F-:<3pt>]\frm{}\restore}\endxy\ ,\ 
\xy<0cm,.9cm>\xymatrix@R=.35cm@C=.35cm{*={} & *={} & *={} & *={} & *={\sss\blacksquare} \ar @{-} [dddd]  \\
*={} & *={} & *={} & *={\sss\blacksquare}   & *={\sss\blacksquare} \ar @{-} [l] \\
*={} & *={} & *={} & *={} \\
*={}   & *={\sss\blacksquare} & *={}  &*={} & *={\sss\blacksquare} \ar @{-} [lll] \\
*={} & & *={\bullet}& *={} & *={\sss\blacksquare} \ar @{-|} [llll] \\  *{}\save "1,3"."5,3"*+<.25cm>[F-:<3pt>]\frm{}\restore}\endxy\ ,\ 
\xy<0cm,.9cm>\xymatrix@R=.35cm@C=.35cm{*={} & *={} & *={} & *={} & *={\sss\blacksquare} \ar @{-} [dddd]  \\
*={} & *={} & *={} & *={\sss\blacksquare}   & *={\sss\blacksquare} \ar @{-} [l] \\
*={} & *={} & *={} & *={} \\
*={}   & *={\sss\blacksquare} & *={}  &*={} & *={\sss\blacksquare} \ar @{-} [lll] \\
*={} & & *={} & *={\bullet} & *={\sss\blacksquare} \ar @{-|} [llll] \\  *{}\save "4,1"."4,5"*+<.25cm>[F-:<3pt>]\frm{}\restore}\endxy}_{\text{not compatible}}.$$
\item[(CC4)] If $u_{1k},u_{jk}\in \FF_q^\times$, with $1<j$ and $S=S_\lambda$, then $S$ sees $u_{1k}$ or $u_{jk}$ if and only if either $u_{jk}$ is not South of the end of a tine but on a tine of $S$, or $u_{jk}$ is South of the end of a tine and weakly North of the next tine to the South,
$$\underbrace
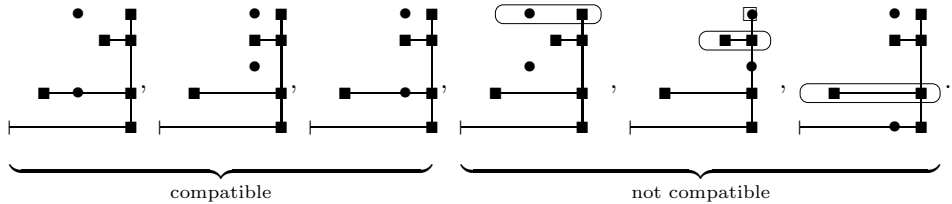
{
\xy<0cm,.9cm>\xymatrix@R=.35cm@C=.35cm{*={} & *={} & *={\bullet} & *={} & *={\sss\blacksquare} \ar @{-} [dddd]  \\
*={} & *={} & *={} & *={\sss\blacksquare}   & *={\sss\blacksquare} \ar @{-} [l] \\
*={} & *={} & *={} & *={} \\
*={}   & *={\sss\blacksquare} & *={\bullet}  &*={} & *={\sss\blacksquare} \ar @{-} [lll] \\
*={} & & *={}& *={} & *={\sss\blacksquare} \ar @{-|} [llll] \\  *{}}\endxy\ ,\ 
\xy<0cm,.9cm>\xymatrix@R=.35cm@C=.35cm{*={} & *={} & *={} & *={\bullet} & *={\sss\blacksquare} \ar @{-} [dddd]  \\
*={} & *={} & *={} & *={\sss\blacksquare}   & *={\sss\blacksquare} \ar @{-} [l] \\
*={} & *={} & *={} & *={\bullet} \\
*={}   & *={\sss\blacksquare} & *={}  &*={} & *={\sss\blacksquare} \ar @{-} [lll] \\
*={} & & *={}& *={} & *={\sss\blacksquare} \ar @{-|} [llll] \\  *{}}\endxy\ ,\ 
\xy<0cm,.9cm>\xymatrix@R=.35cm@C=.35cm{*={} & *={} & *={} & *={\bullet} & *={\sss\blacksquare} \ar @{-} [dddd]  \\
*={} & *={} & *={} & *={\sss\blacksquare}   & *={\sss\blacksquare} \ar @{-} [l] \\
*={} & *={} & *={} & *={} \\
*={}   & *={\sss\blacksquare} & *={}  &*={\bullet} & *={\sss\blacksquare} \ar @{-} [lll] \\
*={} & & *={}& *={} & *={\sss\blacksquare} \ar @{-|} [llll] \\  *{}}\endxy}_{\text{compatible}}\ ,\ 
\underbrace{\xy<0cm,.9cm>\xymatrix@R=.35cm@C=.35cm{*={} & *={} & *={\bullet} & *={} & *={\sss\blacksquare} \ar @{-} [dddd]  \\
*={} & *={} & *={} & *={\sss\blacksquare}   & *={\sss\blacksquare} \ar @{-} [l] \\
*={} & *={} & *={\bullet} & *={} \\
*={}   & *={\sss\blacksquare} & *={}  &*={} & *={\sss\blacksquare} \ar @{-} [lll] \\
*={} & & *={}& *={} & *={\sss\blacksquare} \ar @{-|} [llll] \\  *{}\save "1,2"."1,5"*+<.25cm>[F-:<3pt>]\frm{}\restore}\endxy\ ,\ 
\xy<0cm,.9cm>\xymatrix@R=.35cm@C=.35cm{*={} & *={} & *={} & *={} & *={\bullet\hspace{-.23cm}\ss\square} \ar @{-} [dddd]  \\
*={} & *={} & *={} & *={\sss\blacksquare}   & *={\sss\blacksquare} \ar @{-} [l] \\
*={} & *={} & *={} & *={} & *={\bullet} \\
*={}   & *={\sss\blacksquare} & *={}  &*={} & *={\sss\blacksquare} \ar @{-} [lll] \\
*={} & & *={}& *={} & *={\sss\blacksquare} \ar @{-|} [llll] \\  *{}\save "2,3"."2,5"*+<.25cm>[F-:<3pt>]\frm{}\restore}\endxy\ ,\ 
\xy<0cm,.9cm>\xymatrix@R=.35cm@C=.35cm{*={} & *={} & *={} & *={\bullet} & *={\sss\blacksquare} \ar @{-} [dddd]  \\
*={} & *={} & *={} & *={\sss\blacksquare}   & *={\sss\blacksquare} \ar @{-} [l] \\
*={} & *={} & *={} & *={} \\
*={}   & *={\sss\blacksquare} & *={}  &*={} & *={\sss\blacksquare} \ar @{-} [lll] \\
*={} & & *={} & *={\bullet} & *={\sss\blacksquare} \ar @{-|} [llll] \\  *{}\save "4,1"."4,5"*+<.25cm>[F-:<3pt>]\frm{}\restore}\endxy}_{\text{not compatible}}.$$
\end{enumerate}

From this point of view, Theorem \ref{TriangleCharacterFormula} translates to the following corollary.

\begin{corollary}  \label{CombCharacterFormula} Suppose $u\in U_{(m)}$ has at most two nonzero superdiagonal entries $u_{1k}, u_{jk}\in \FF_q$, for some $1\leq k\leq n$.  For $\lambda\in\cT_{(m)}$, suppose $G_\lambda$ has one connected component $S$.  Then 
 \vspace{.25cm}
 
 \noindent (a) The character degree
 $$\chi^\lambda(1)=\left\{\begin{array}{ll} |\{i<j\in \cP_{(m)}\ \mid\ \text{$\lambda_{ik}\neq 0$, for $k>j>i>1$, $\lambda_{ik'}\neq 0$ implies $k'\geq k$}\}|, & \text{if $\wt(\lambda)=0$,}\\
|\{i<j\in \cP_{(m)}\ \mid\ \text{$\lambda_{ik}\neq 0$, for $k>j>i$, $\lambda_{ik'}\neq 0$ implies $k'\geq k$}\}|, & \text{if $\wt(\lambda)=1$.} \end{array}\right.$$

\vspace{.25cm}

\noindent (b)  The character 
$$\chi^\lambda(u)=0$$
unless column $k$ of $u$ and $S$ are comb compatible and in condition (CC4) if $u_{1k}$ sees $S$ and $u_{jk}$ is not  strictly South and weakly East of the end of a tine, 
\begin{equation}\label{CombSpecialCondition}\xy<0cm,.9cm>\xymatrix@R=.35cm@C=.35cm{*={} & *={} & *{\ss u_{1k}} & *={} & *{\ss\lambda_{1l}} \ar @{-} [ddd]  \\
*={} & *={} & *={} & *={\sss\blacksquare}   & *={\sss\blacksquare} \ar @{-} [l] \\
*={} & *={} & *={} & *={} \\
*={}   & *={\sss\blacksquare} & *{\ss u_{jk}} \ar @{-} [l]  &*={} & *{\ss\lambda_{jl}} \ar @{-} [ll] \ar @{-} [d] \\
*={} & & *={}& *={} & *={\sss\blacksquare} \ar @{-|} [llll] \\  *{}}\endxy,\qquad\text{then}\qquad\text{$u_{1k}\lambda_{1l} + u_{jk}\lambda_{il}=0$.} \end{equation}

\vspace{.25cm}

\noindent (c)  If $\chi^\lambda(u)\neq 0$, then
$$\chi^\lambda(u)= \chi^\lambda(1)\frac{\theta(u_{1k}\lambda_{1k}+u_{jk}\lambda_{jk})}{q^{c(u,\lambda)}}\prod_{i<j<l\atop \lambda_{il},\lambda_{jl}\in \FF_q^\times}\theta(-\lambda_{il}^{-1} \lambda_{ik}(u_{1k}\lambda_{1l}+u_{jk}\lambda_{jl}))$$
where 
$$c(u,\lambda)=\left\{\begin{array}{ll}
|\{l>k\ \mid\ \lambda_{il}\neq 0, \text{ for some $i<j$}\}|, & \text{if $u_{jk}\neq 0$, $j>m$,}\\
|\{l>k\ \mid\ \lambda_{il}\neq 0, \text{ for some $i<j$}\}|,& \text{if  $u_{jk},\lambda_{ij'} \lambda_{il}\in \FF_q^\times$ with $j\leq m$, $i<j'<k<l$,}\\
|\{l>k\ \mid\ \lambda_{il}\neq 0, \text{ for some $i<j$}\}|-1, & \text{otherwise.}
\end{array}\right.$$
\end{corollary}

\subsection{A character formula for path representatives}

For $\lambda\in \cZ_{(m)}\cup\cT_{(m)}$, let $S_\lambda$, $\lc(\lambda)$, $\br(\lambda)$, and $\wt(\lambda)$ be as in the previous section.  For $\lambda\in \cZ_{(m)}$, order the vertices of $S_\lambda$ starting with the vertex in the first row and proceeding along the path to the vertex  in position $(\br(\lambda),\cend(\lambda))$.  For example, if
$$\lambda=
\left(\begin{array}{cccccc}  0 & 0 & 0 & 0 & 0 & a\\
				            0 & 0 & 0 & 0 & c & b\\
				            0 & 0 & 0 & e & 0 & 0\\
				            0 & 0 & 0 & 0 & d & 0\\
				            0 & 0 & 0 & 0 & 0 & 0\\
				            0 & 0 & 0 & 0 & 0 & 0\end{array}\right),\quad\text{then}\quad
S_\lambda=\xy<0cm,.9cm>\xymatrix@R=.35cm@C=.35cm{ *={} & *{a} \ar @{-} [d]  \\
 *{c} \ar @{-} [dd]   & *{b} \ar @{-} [l] \\
 *={} \\
 *{d}}\endxy,\quad \begin{array}{l} \cend(\lambda)=5,\\ \br(\lambda)=4,\\ \wt(\lambda)=0.\end{array}$$

Before translating from comb representatives to path representatives, we will add some decorations to the graphs $G_{u-1}$ and $G_\lambda$ of Section \ref{TheInterpolatingGroups}.  We will again suppress the values of the vertices and distinguish between $G_{u-1}$ and $G_\lambda$ by the following conventions,
$$\begin{array}{|l|c|c|} \hline & G_{u-1} & G_\lambda  \\  \hline
 \text{Vertices} & \bullet & \sss\blacksquare\\  \hline
 \text{Edges} & \xymatrix@C=.5cm@R=.2cm{*={\bullet}\ar @{:} [r] & *={\bullet}} & \xymatrix@C=.5cm@R=.2cm{*={\sss\blacksquare}\ar @{-} [r] & *={ \sss\blacksquare} } \\
 \hline\end{array}$$

If $S_\lambda\neq \emptyset$, then add an additional edge from the vertex in position $(\br(\lambda),\cend(\lambda))$.  If $\wt(\lambda)=0$ and $|S_\lambda|\neq 1$ extend the edge West until it reaches just pasat the $(m+1)$th column.  If $\wt(\lambda)=1$ or $|S_\lambda|=1$, then extend the edge South until just past the $m$th row.   For example,
$$\left(\begin{array}{c|c} 
\begin{array}{ccc} 1 & & \\ & \ddots & \\ & & 1\end{array} & \xy<0cm,.7cm>\xymatrix@R=.35cm@C=.35cm{*={} & *={} & *={} & *={} & *={\sss \blacksquare} \ar @{-} [d]  \\
*={} & *={} & *={} & *={\sss \blacksquare} \ar @{-} [d] \ar @{} [ddll]|(.7){\adots}  & *={\sss \blacksquare} \ar @{-} [l] \\
*={} & *={}  & *={\sss \blacksquare} & *={\sss \blacksquare} \ar @{-} [l] \\
 *={\sss \blacksquare} \ar @{-} [d] & *={\sss \blacksquare} \ar @{-} [l] & *={}\\ *={\sss \blacksquare}\save[]-<.5cm,0cm> *={} \ar @{|-} [] \\ *{}}\endxy \\ \hline
0 &\begin{array}{ccc} 1 & & \\ & \ddots & \\ & & 1\end{array} 
\end{array}\right),\ 
\left(\begin{array}{c|c} 
\begin{array}{ccc} 1 & & \\ & \ddots & \\ & & 1\end{array} & \xy<0cm,.7cm>\xymatrix@R=.35cm@C=.35cm{*={} & *={} & *={} & *={} & *={\sss \blacksquare} \ar @{-} [d]  \\
*={} & *={} & *={} & *={\sss \blacksquare} \ar @{-} [d] \ar @{} [ddll]|(.7){\adots}  & *={\sss \blacksquare} \ar @{-} [l] \\
*={} & *={}  & *={\sss \blacksquare} & *={\sss \blacksquare} \ar @{-} [l] \\
 *={\sss \blacksquare} 
& *={\sss \blacksquare} \ar @{-} [l] & *={}\\ *={}\save[]-<0cm,.15cm> \ar @{|-} [u] \restore \\ *{}}\endxy \\ \hline
0 &\begin{array}{ccc} 1 & & \\ & \ddots & \\ & & 1\end{array} 
\end{array}\right),\ \text{or}\ \left(\xy<0cm,1cm>\xymatrix@R=.7cm@C=.4cm{*={} & *={} & *={} & *={} & *={\sss \blacksquare}  \ar @{-|} [d]  \\
*={} & *={} & *={}  & *={} & *={} & *{}\\
*{} \\ *{}}\endxy\right).$$
A \emph{bottom corner} of $\lambda$ is a vertex $v$ in $S_\lambda$ with a horizontal edge extending West of $v$.  A \emph{top corner}  of $\lambda$ is a vertex $v$ in $S_\lambda$ with a vertical edge extending South of $v$.
All vertices of $G_\lambda$ which are not in $S_\lambda$ are considered to be both top and bottom corners.

Similarly, if $u\in \cZ_{(m)}^\vee$, then $G_{u-1}$ has at most one connected component $S_u$ that has a vertex in the first row. Order the vertices of $S_u$ starting with the vertex in the first row, and proceeding along the path.  If $|S_u|>1$, add an edge to $S_u$ by extending an edge from the last vertex in the opposite direction of the previous edge (either East or North). Furthermore, if $|S_u|>1$, then view the first edge as not being in the same plane as the matrix, so it no longer crosses any edges that are in the plane of the matrix. Thus, $S_u$ will be of the form,
$$\left(\begin{array}{c|c} 
\begin{array}{ccc} 1 & & \\ & \ddots & \\ & & 1\end{array} & \xy<0cm,.7cm>\xymatrix@R=.35cm@C=.35cm{*={\bullet} \ar @{:} @(ul,d) [dddd] & *={} & *={} & *={}  \\
*={} & *={} & *={} & *={} & *={\bullet} \ar @{:} [d] \ar @{} [ddll]|(.7){\adots}  \save[]+<.5cm,0cm> *={} \ar @{|:} []\restore \\
*={} & *={} & *={}  & *={\bullet} & *={\bullet} \ar @{:} [l] \\
*={}& *={\bullet} \ar @{:} [r] & *={\bullet} \\ *={\bullet} \ar@{:} [r] & *={\bullet}\ar @{:} [u]\\ *{}}\endxy \\ \hline
0 &\begin{array}{ccc} 1 & & \\ & \ddots & \\ & & 1\end{array} 
\end{array}\right)
\qquad \text{or} \qquad 
\left(\begin{array}{c|c} 
\begin{array}{ccc} 1 & & \\ & \ddots & \\ & & 1\end{array} & \xy<0cm,.7cm>\xymatrix@R=.35cm@C=.35cm{*={\bullet} \ar @{:} @(ul,d)  [dddd] & *={} & *={} & *={}  & *={}  & *={}\save[]+<0cm,.15cm> \ar @{|:} [d] \restore\\
*={} & *={} & *={} & *={} & *={\bullet} \ar @{:} [d] \ar @{} [ddll]|(.7){\adots}  & *={\bullet}\ar @{:}[l] \\
*={} & *={} & *={}  & *={\bullet} & *={\bullet} \ar @{:} [l] \\
*={}& *={\bullet} \ar @{:} [r] & *={\bullet} \\ *={\bullet} \ar@{:} [r] & *={\bullet}\ar @{:} [u]\\ *{}}\endxy \\ \hline
0 &\begin{array}{ccc} 1 & & \\ & \ddots & \\ & & 1\end{array} 
\end{array}\right).$$
The \emph{left corners} of $u$ are the leftmost nonzero entries in the rows of $u-1$.  The \emph{right corners} of $u$ are the rightmost nonzero entries in the rows of $u-1$.  

Left and right corners \emph{see} North in their column and East in their row, while top and bottom corners \emph{see} South in their column and West in their row (in both cases they do not see the location they are in).  That is,
$$\xy<0cm,.25cm>\xymatrix@R=.5cm@C=.5cm{*={} & *={}\\ *={\bullet} \ar [r] \ar [u]  & *={}}\endxy \qquad\text{and}\qquad 
\xy<0cm,.25cm>\xymatrix@R=.5cm@C=.5cm{*={} & *={\sss\blacksquare}\ar [d] \ar [l]\\ *={} & *={}}\endxy\ .$$
Connected components $S$ of $G_u$ and $T$ of $G_\lambda$ \emph{see one-another} if when one superimposes their matrices, a corner of $S$ sees a corner of $T$.

Fix $u\in \cZ_{(m)}^\vee$ and $\lambda\in \cZ_{(m)}$ with a connected component $S$ of $G_u$ and $T$ of $G_\lambda$.  The components $S$ and $T$ are \emph{path compatible} if the following conditions are satisfied.
\begin{enumerate}
\item[(PC1)]  If $S\neq S_u$ and $T\neq S_\lambda$, then $S$ cannot see $T$,
$$\underbrace{\begin{array}{cc} & \sss\blacksquare\\ \bullet & \end{array},\ 
\begin{array}{cc} \bullet & \\ \sss\blacksquare & \end{array},\ 
\begin{array}{cc} & \\ \sss\blacksquare & \bullet \end{array}}_{\text{compatible}},\ \underbrace{
\begin{array}{cc}\bullet &\sss\blacksquare \\  &  \end{array},\ 
\begin{array}{cc} & \sss\blacksquare  \\& \bullet \end{array}
}_{\text{not compatible}}.$$
\item[(PC2)]  If $S=S_u$ and $T\neq S_\lambda$, then $S$ sees $T$ if and only if $T$ touches a vertical edge of $S$ and no left corner of $S$ sees $T$.
$$\underbrace{\xy<0cm,.9cm>\xymatrix@R=.35cm@C=.35cm{*={\bullet} \ar @{:} @(ul,d) [dddd] & *={}  \\
*={} & *={} & *={\bullet} \ar @{:|} [r] & *={}  \\
*={} & *={} &  *={}   \\
*={}& *={\bullet} \ar @{:} [r] & *={\bullet} \ar @{:} [uu]\\ *={\bullet} \ar@{:} [r] & *={\bullet}\ar @{:} [u]\\ *={} & *={} & *={\sss\blacksquare}}\endxy,\qquad 
\xy<0cm,.9cm>\xymatrix@R=.35cm@C=.35cm{*={\bullet} \ar @{:} @(ul,d) [dddd] & *={}   \\
*={} & *={} & *={\bullet} \ar @{:|} [r] & *={}   \\
*={} & *={} &  *={\sss\blacksquare}   \\
*={}& *={\bullet} \ar @{:} [r] & *={\bullet} \ar @{:} [uu]\\ *={\bullet} \ar@{:} [r] & *={\bullet}\ar @{:} [u]\\ *={} }\endxy,\qquad 
\xy<0cm,.9cm>\xymatrix@R=.35cm@C=.35cm{*={\bullet} \ar @{:} @(ul,d) [dddd] & *={}  \\
*={} & *={} & *={\bullet\hspace{-.23cm}{\ss\square}} \ar @{:|} [r] & *={}   \\
*={} & *={} &  *={}   \\
*={}& *={\bullet} \ar @{:} [r] & *={\bullet} \ar @{:} [uu]\\ *={\bullet} \ar@{:} [r] & *={\bullet}\ar @{:} [u]\\ *={} }\endxy}_{\text{compatible}},\qquad 
\underbrace{\xy<0cm,.9cm>\xymatrix@R=.35cm@C=.35cm{*={\bullet} \ar @{:} @(ul,d) [dddd] & *={}  \\
*={} & *={} & *={\bullet} \ar @{:|} [r] & *={}   \\
*={} & *={} &  *={}   \\
*={}& *={\bullet} \ar @{:} [r] & *={\bullet} \ar @{:} [uu]\\ *={\bullet} \ar@{:} [r] & *={\bullet}\ar @{:} [u] & *={\sss\blacksquare} \\ *={} \save "5,1"."5,3"*+<.25cm>[F-:<3pt>]\frm{}\restore}\endxy,\qquad 
\xy<0cm,.9cm>\xymatrix@R=.35cm@C=.35cm{*={\bullet} \ar @{:} @(ul,d) [dddd] & *={}  \\
*={} & *={} & *={\bullet} \ar @{:|} [r] & *={}   \\
*={} & *={} &  *={}  & *={} \\
*={}& *={\bullet} \ar @{:} [r] & *={\bullet\hspace{-.23cm}{\ss \square}} \ar @{:} [uu]& *={}\\ *={\bullet} \ar@{:} [r] & *={\bullet}\ar @{:} [u]\\ *={} \save "4,1"."4,3"*+<.25cm>[F-:<3pt>]\frm{}\restore}\endxy,\qquad 
\xy<0cm,.9cm>\xymatrix@R=.35cm@C=.35cm{*={\bullet} \ar @{:} @(ul,d) [dddd] & *={}  \\
*={} & *={\sss\blacksquare} & *={\bullet} \ar @{:|} [r] & *={}   \\
*={} & *={} &  *={}   \\
*={}& *={\bullet} \ar @{:} [r] & *={\bullet} \ar @{:} [uu]\\ *={\bullet} \ar@{:} [r] & *={\bullet}\ar @{:} [u]\\ *={}\save "1,2"."4,2"*+<.25cm>[F-:<3pt>]\frm{}\restore }\endxy}_{\text{not compatible}}.$$
\item[(PC3)] If $S\neq S_u$ and $T=S_\lambda$, then $S$ sees $T$ if and only if $S$ touches a vertical edge of $T$ and no bottom corner of $T$  sees $S$.
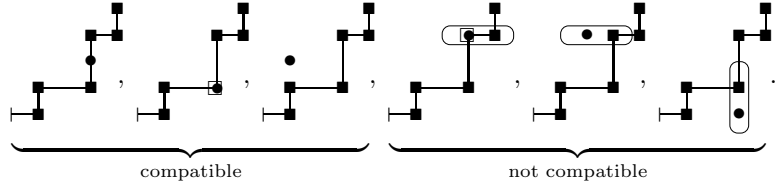
$$\underbrace{\xy<0cm,.9cm>\xymatrix@R=.35cm@C=.35cm{*={} & *={} & *={} & *={} & *={\sss\blacksquare} \ar @{-} [d]  \\
*={} & *={} & *={} & *={\sss\blacksquare} \ar @{-} [dd]   & *={\sss\blacksquare} \ar @{-} [l] \\
*={} & *={} & *={} & *={\bullet} \\
*={}   & *={\sss\blacksquare} \ar @{-} [d]& *={}  & *={\sss\blacksquare} \ar @{-} [ll] \\
*={} & *={\sss\blacksquare} \ar @{-|} [l] & *={}\\  *{}}\endxy,\ 
\xy<0cm,.9cm>\xymatrix@R=.35cm@C=.35cm{*={} & *={} & *={} & *={} & *={\sss\blacksquare} \ar @{-} [d]  \\
*={} & *={} & *={} & *={\sss\blacksquare} \ar @{-} [dd]   & *={\sss\blacksquare} \ar @{-} [l] \\
*={} & *={} & *={} & *={} \\
*={}   & *={\sss\blacksquare} \ar @{-} [d]& *={}  & *={\bullet\hspace{-.23cm}{\ss\square}} \ar @{-} [ll] \\
*={} & *={\sss\blacksquare} \ar @{-|} [l] & *={}\\  *{}}\endxy,\ 
\xy<0cm,.9cm>\xymatrix@R=.35cm@C=.35cm{*={} & *={} & *={} & *={} & *={\sss\blacksquare} \ar @{-} [d]  \\
*={} & *={} & *={} & *={\sss\blacksquare} \ar @{-} [dd]   & *={\sss\blacksquare} \ar @{-} [l] \\
*={} & *={\bullet} & *={} & *={} \\
*={}   & *={\sss\blacksquare} \ar @{-} [d]& *={}  & *={\sss\blacksquare} \ar @{-} [ll]  \\
*={} & *={\sss\blacksquare} \ar @{-|} [l] & *={}& *={} \\  *{}}\endxy}_{\text{compatible}},\ 
\underbrace{\xy<0cm,.9cm>\xymatrix@R=.35cm@C=.35cm{*={} & *={} & *={} & *={} & *={\sss\blacksquare} \ar @{-} [d]  \\
*={} & *={} & *={} & *={\bullet\hspace{-.23cm}\ss\square} \ar @{-} [dd]   & *={\sss\blacksquare} \ar @{-} [l] \\
*={} & *={} & *={} & *={} \\
*={}   & *={\sss\blacksquare} \ar @{-} [d]& *={}  & *={\sss\blacksquare} \ar @{-} [ll] \\
*={} & *={\sss\blacksquare} \ar @{-|} [l] & *={}\\  *{} \save "2,3"."2,5"*+<.25cm>[F-:<3pt>]\frm{}\restore}\endxy,\ 
\xy<0cm,.9cm>\xymatrix@R=.35cm@C=.35cm{*={} & *={} & *={} & *={} & *={\sss\blacksquare} \ar @{-} [d]  \\
*={} & *={} & *={\bullet} & *={\sss\blacksquare} \ar @{-} [dd]   & *={\sss\blacksquare} \ar @{-} [l] \\
*={} & *={} & *={} & *={} \\
*={}   & *={\sss\blacksquare} \ar @{-} [d]& *={}  & *={\sss\blacksquare} \ar @{-} [ll] \\
*={} & *={\sss\blacksquare} \ar @{-|} [l] & *={}\\  *{} \save "2,2"."2,4"*+<.25cm>[F-:<3pt>]\frm{}\restore}\endxy,\ 
\xy<0cm,.9cm>\xymatrix@R=.35cm@C=.35cm{*={} & *={} & *={} & *={} & *={\sss\blacksquare} \ar @{-} [d]  \\
*={} & *={} & *={} & *={\sss\blacksquare} \ar @{-} [dd]   & *={\sss\blacksquare} \ar @{-} [l] \\
*={} & *={} & *={} & *={} \\
*={}   & *={\sss\blacksquare} \ar @{-} [d] & *={}  & *={\sss\blacksquare} \ar @{-} [ll] \\
*={} & *={\sss\blacksquare} \ar @{-|} [l] &  *={} & *={\bullet}\\  *{} \save "3,4"."5,4"*+<.25cm>[F-:<3pt>]\frm{}\restore}\endxy}_{\text{not compatible}}. $$ 
\item[(PC4)] If $S=S_u$ and $T=S_\lambda$, then $S$ sees $T$ if and only if $T$ is never strictly South of $S$; $S$ ends weakly East of  the beginning of $T$;  and left corners of $S$ and bottom corners of $T$ only see one-another horizontally. 
$$\underbrace
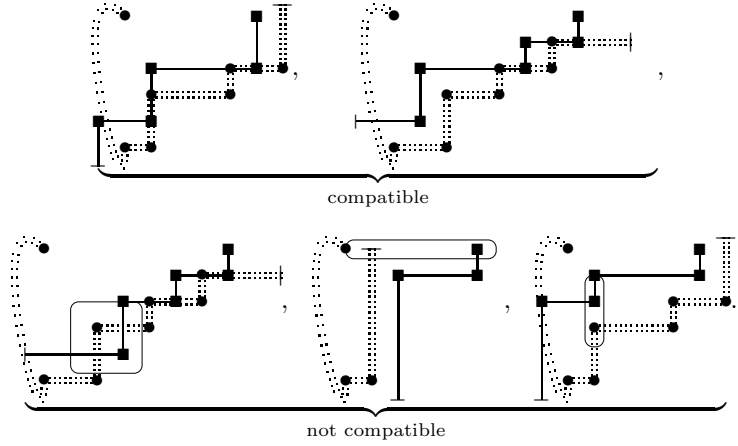
{\xy<0cm,.7cm>\xymatrix@R=.35cm@C=.35cm{*={} & *={\bullet} \ar @{:} @(ul,d) [ddddd] & *={} & *={} & *={} & *={} & *={\sss \blacksquare} \ar @{-} [dd]   & *={}\save[]+<0cm,.15cm> \ar @{|:} [dd] \restore\\
*={} & *={} & *={} & *={} & *={} & *={}  & *={}  & *={}\\
*={} & *={} &*={\sss \blacksquare} \ar @{-} [dd] & *={} & *={} & *={\bullet} \ar @{:} [d] & *={\sss \blacksquare} \ar@{-} [llll] & *={\bullet}\ar @{:}[ll] \\
*={} & *={} & *={\bullet} \ar @{:} [dd] & *={} & *={} & *={\bullet} \ar @{:} [lll] & *={}  \\
*={\sss\blacksquare}  & *={} & *={\sss\blacksquare} \ar @{-} [ll] \\ 
*={}\save[]-<0cm,.25cm> \ar @{|-} [u]\restore & *={\bullet} \ar@{:} [r] & *={\bullet}}\endxy\ , 
 \qquad
 \xy<0cm,.7cm>\xymatrix@R=.35cm@C=.35cm{ *={\bullet} \ar @{:} @(ul,d) [ddddd] & *={} & *={} & *={} & *={} & *={} & *={} &  *={\sss \blacksquare} \ar @{-} [d]  \\
 *={} & *={} & *={} & *={} & *={}  & *={\sss\blacksquare} \ar @{-} [d]  & *={\bullet} \ar @{:} [d]  \ar @{:|} [rrr] & *={\sss\blacksquare} \ar @{-} [ll] & *= {} & *={} &*={} \\
 *={} &*={\sss \blacksquare} \ar @{-} [dd] & *={} & *={} & *={\bullet} \ar @{:} [d] & *={\sss \blacksquare} \ar@{-} [llll] & *={\bullet}\ar @{:}[ll] \\
  *={} & *={}  & *={\bullet} \ar @{:} [dd]  & *={} & *={\bullet} \ar @{:} [ll] & *={}  \\
 *={} \save[]-<.5cm,0cm> *={} \ar @{|-} [r]\restore & *={\sss \blacksquare}  \\ 
 *={\bullet} \ar@{:} [rr] & *={} & *={\bullet}}\endxy}_{\text{compatible}},$$ 
$$\underbrace{ \xy<0cm,.7cm>\xymatrix@R=.35cm@C=.35cm{ *={\bullet} \ar @{:} @(ul,d) [ddddd] & *={} & *={} & *={} & *={} & *={} & *={} &  *={\sss \blacksquare} \ar @{-} [d]  \\
 *={} & *={} & *={} & *={} & *={}  & *={\sss\blacksquare} \ar @{-} [d]  & *={\bullet} \ar @{:} [d]  \ar @{:|} [rrr] & *={\sss\blacksquare} \ar @{-} [ll] & *={} & *={}  \\
 *={} & *={} & *={} & *={\sss \blacksquare} \ar @{-} [dd]   & *={\bullet} \ar @{:} [d] & *={\sss \blacksquare} \ar@{-} [ll] & *={\bullet}\ar @{:}[ll] \\
  *={} & *={}  & *={\bullet} \ar @{:} [dd]  & *={} & *={\bullet} \ar @{:} [ll] & *={}  \\
 *={} \save[]-<.25cm,0cm> *={} \ar @{|-} [rrr]\restore &*={} & *={} & *={\sss \blacksquare}  \\ 
 *={\bullet} \ar@{:} [rr] & *={} & *={\bullet}\save "3,2"."5,4"*+<.25cm>[F-:<3pt>]\frm{}\restore}\endxy,
\quad 
\xy<0cm,.7cm>\xymatrix@R=.35cm@C=.35cm{*={} & *={\bullet} \ar @{:} @(ul,d) [ddddd] & *={} & *={} & *={} & *={} & *={\sss \blacksquare} \ar @{-} [d]   & *={}\\
*={} & *={} & *={} & *={\sss \blacksquare} & *={} & *={}  & *={\sss \blacksquare} \ar@{-} [lll]  & *={}\\
*={} & *={} & *={}  & *={} & *={} & *={} & *={} & *={} \\
*={} & *={} & *={} & *={} & *={} & *={} & *={}  \\
*={}  & *={} & *={} \\ 
*={} & *={\bullet} \ar@{:} [r] & *={\bullet}\ar @{:|}[uuuuu] & *={} \save[]-<0cm,.25cm> \ar @{|-} [uuuu]\restore  \save "1,2"."1,7"*+<.25cm>[F-:<3pt>]\frm{}\restore}\endxy,\quad
\xy<0cm,.7cm>\xymatrix@R=.35cm@C=.35cm{*={} & *={\bullet} \ar @{:} @(ul,d) [ddddd] & *={} & *={} & *={} & *={} & *={\sss \blacksquare} \ar @{-} [d]   & *={}\save[]+<0cm,.15cm> \ar @{|:} [dd] \restore\\
*={} & *={} & *={\sss \blacksquare} \ar @{-} [d] & *={} & *={} & *={}  & *={\sss \blacksquare} \ar@{-} [llll]  & *={}\\
*={\sss\blacksquare} & *={} & *={\sss\blacksquare} \ar @{-} [ll]  & *={} & *={} & *={\bullet} \ar @{:} [d] & *={} & *={\bullet}\ar @{:}[ll] \\
*={} & *={} & *={\bullet} \ar @{:} [dd] & *={} & *={} & *={\bullet} \ar @{:} [lll] & *={}  \\
*={}  & *={} & *={} \\ 
*={}\save[]-<0cm,.25cm> \ar @{|-} [uuu]\restore & *={\bullet} \ar@{:} [r] & *={\bullet}\save "2,3"."4,3"*+<.25cm>[F-:<3pt>]\frm{}\restore}\endxy}_{\text{not compatible}}.$$
\end{enumerate}
Note that (PC1)-(PC4) are translations of (CC1)-(CC4) via the correspondence (\ref{PathToCombCharacter}).

\begin{corollary}\label{PathCharacterFormula} Let $u\in U_{(m)}$ be such that $u-1\in \cZ_{(m)}^\vee$, and let $\lambda\in \cZ_{(m)}$.  Then
\begin{enumerate}
\item[(a)] The character
$$\chi^\lambda(u)=0$$
unless the connected components of $G_{u-1}$ and $G_\lambda$ are pairwise path compatible, and in the superimposed matrices, every
\begin{equation}\label{HorizontalTouch} \xy<0cm,.5cm>\xymatrix@R=.5cm@C=1cm{*={} & *={} & *={}\\
*={} & *{x_j} \ar @{:>} [rr]|(.55){\hspace{.4cm}} \ar @{} [r]^{\text{strict}} \ar @{:>}[d]  & *{y_k} \ar  [u]\ar [ll]|(.55){\hspace{.4cm}}  & *={}\\
*={} & *={} }\endxy\qquad \text{implies} \qquad \mathrm{bag}(x_j)\mathrm{bag}(y_k)=1.\end{equation}
\item[(b)] The character degree
$$\chi^\lambda(1)=\left\{\begin{array}{ll} |\{i<j\in \cP_{(m)}\  \mid\ \text{$\lambda_{ik}$ is a bottom corner for some $k>j$}\}|, & \text{if $\wt(\lambda)=0$,}\\  |\{i<j\in \cP_{(m)}\  \mid\ \text{$\lambda_{ik}$ is a top corner for some $k>j$}\}|, & \text{if $\wt(\lambda)=1$}.\end{array}\right.$$
\item[(c)]  If $\chi^\lambda(u)\neq 0$, then
$$\chi^\lambda(u)=\chi^\lambda(1) \theta(\lambda(u-1))\prod_{\text{left corners}\atop u_{jk}}\frac{1}{q^{\#\left\{\begin{array}{@{}c@{}}\ss\text{bottom corners $\lambda_{il}$}\\ \ss\text{with $i<j<k<l$}\end{array}\right\}}}\prod_{
\xymatrix@R=.35cm@C=.35cm{*={} & *{\ss y_k} \ar @{-} [d] & *={}\\
*={} & *={\bullet\hspace{-.2cm}{\ss\Box}} \ar [l] \ar @{:>} [r] &  *={}\\
*={} & *{\ss x_j} \ar @{:} [u] } 
\text{or} 
\xymatrix@R=.35cm@C=.35cm{*={} & *{\ss y_1} \ar @{-} [d] & *={}\\
*={} & *={\sss{\blacksquare}} \ar [l]  &  *={}\\
*={} & *{\ss x_j} \ar @{:|} [uu] } }\theta(\mathrm{bag}(x_j)\mathrm{bag}(y_k)).$$
\end{enumerate}
\end{corollary}
\begin{proof}  This corollary follows directly from Corollary \ref{CombCharacterFormula} with the following observations, using (\ref{PathToCombCharacter}).  

(a) If a bottom corner of $T$ sees a left corner of $S$ horizontally, then we are in the situation of (\ref{HorizontalTouch}), so
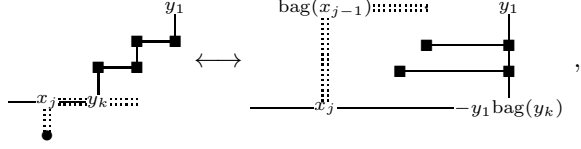
$$ \xy<0cm,.7cm>\xymatrix@R=.35cm@C=.35cm{  *={} & *={} & *={} & *={} & *+<1pt>{\ss y_1} \ar @{-} [d]   \\
*={} & *={} & *={} & *={\sss\blacksquare} \ar @{-} [d] & *={\sss\blacksquare} \ar @{-} [l]   \\
*={} & *={} & *={\sss\blacksquare}  \ar @{-} [d] & *={\sss\blacksquare} \ar @{-} [l]   \\
*={} & *+<1pt>{\ss x_j} \ar @{:} [d] \ar @{:} [r] \ar @{-} [l] & *+<1pt>{\ss y_k} \ar @{-} [l]\ar @{:} [r]  & *={} \\
*={} & *={\bullet}}\endxy
\longleftrightarrow 
 \xy<0cm,.7cm>\xymatrix@R=.35cm@C=.35cm{  *={} & *+<1pt>{\ss\bag(x_{j-1})} \ar @{:} [ddd]\ar @{:} [rr] & *={} & *={\cdot} & *+<1pt>{\ss y_1} \ar @{-} [ddd]   \\
*={} & *={} & *={} & *={\sss\blacksquare}  & *={\sss\blacksquare} \ar @{-} [l]   \\
*={} & *={} & *={\sss\blacksquare}   & *={} & *={\sss\blacksquare} \ar @{-} [ll]   \\
*={} & *+<1pt>{\ss x_j}   \ar @{-} [l] & *={} & *={}   & *+<1pt>{\ss -y_1\bag(y_k)}\ar @{-} [lll]\\
*={} & *={}}\endxy\ ,
 $$
so the comb representations of  $\lambda$ and $u$ must satisfy condition (\ref{CombSpecialCondition}).  However, this is equivalent to $\bag(x_j)\bag(y_k)=1$.  Thus, Corollary \ref{CombCharacterFormula} (a) is satisfied if and only if Corollary \ref{PathCharacterFormula} (a) is satisfied.  
 
 (b) is straight-forward translation of the combinatorics.  
 
 (c) First note that
$$\prod_{\text{left corners}\atop u_{jk}}\frac{1}{q^{\#\left\{\begin{array}{@{}c@{}}\ss\text{bottom corners $\lambda_{il}$}\\ \ss\text{with $i<j<k<l$}\end{array}\right\}}}=\prod_{k} \frac{1}{q^{c(u,\lambda)}}.$$
If $\chi^\lambda(u)\neq 0$, and we have no configurations of the form
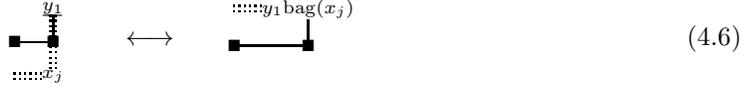
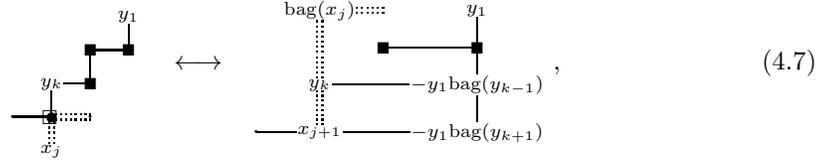
\begin{align}\label{DegenerateEnding} \xy<0cm,.4cm>\xymatrix@R=.35cm@C=.35cm{  
 *={} & *+<1pt>{\ss y_1} \ar @{-} [d]   \\
 *={\sss\blacksquare}  & *={\sss\blacksquare} \ar @{-} [l]   \\
*={} & *+<1pt>{\ss x_j} \ar @{:|} [uu] \ar @{:} [l] }\endxy
&\qquad \longleftrightarrow\qquad
\xy<0cm,.4cm>\xymatrix@R=.35cm@C=.35cm{  
 *={} & *+<1pt>{\ss y_1 \bag(x_j)} \ar @{-} [d]  \ar @{:} [l] \\
 *={\sss\blacksquare}  & *={\sss\blacksquare} \ar @{-} [l]   \\
 *={}}\endxy\\
\xy<0cm,.7cm>\xymatrix@R=.35cm@C=.35cm{  
*={} & *={} & *={} & *+<1pt>{\ss y_1}\ar @{-} [d]   \\
*={} & *={} & *={\sss\blacksquare}  \ar @{-} [d] & *={\sss\blacksquare} \ar @{-} [l]   \\
*={} & *+<1pt>{\ss y_k} \ar @{-} [d] & *={\sss\blacksquare} \ar @{-} [l]  \\
*={\cdot} & *={\bullet\hspace{-.23cm}\ss\square} \ar @{:} [r]\ar @{-} [l] & *={\cdot}\\
*={} & *+<1pt>{\ss x_j} \ar @{:} [u]}\endxy
&\quad \longleftrightarrow\quad
\xy<0cm,.7cm>\xymatrix@R=.35cm@C=.35cm{  
*={} & *+<1pt>{\ss \bag(x_j)} \ar @{:} [dd]\ar @{:} [r]  & *={} & *+<1pt>{\ss y_1} \ar @{-} [dd]   \\
*={} & *={} & *={\sss\blacksquare}  & *={\sss\blacksquare} \ar @{-} [l]   \\
*={} & *+<1pt>{\ss y_k} \ar @{:}[d] & *={} & *+<1pt>{\ss-y_1\bag(y_{k-1})} \ar @{-} [ll] \ar @{-} [d] \\
*={\cdot} & *+<1pt>{\ss x_{j+1}} \ar @{-} [l] & *={} & *+<1pt>{\ss -y_1\bag(y_{k+1})} \ar @{-} [ll] \\ 
*={} & *={} }\endxy\ , \label{PathsTouching}\end{align}
then both character formulas are equal.  If (\ref{DegenerateEnding}) occurs then both the path character formula and the comb character formula get a factor of 
$$\theta(\bag(x_jy_1)).$$
If (\ref{PathsTouching}) occurs then the path character formula gets a factor 
$$\theta(y_{k+1}x_{j+1}) \theta(\bag(x_j)\bag(y_k)),$$
and the comb character formula gets a factor of 
$$\theta\left(-(-y_1\bag(y_{k-1})^{-1}y_k\big(\bag(x_j)y_1+x_{j+1}(-y_1\bag(y_{k+1})\big)\right),$$
However, a simple computation shows that these two quantities are equal.  Thus, the character formulas for the two types of representatives are equal.
\end{proof}

\subsection{Example}

 Let $n=7$, $m=4$,
$$u=
\left(\begin{array}{ccccccc}
1 & 0 & 0 & 0 & a & 0 & 0\\
0 & 1 & 0 & 0 & 0 & d & 0\\
0 & 0 & 1 & e & 0 & 0 & 0\\
0 & 0 & 0 & 1 & b & c & 0\\
0 & 0 & 0 & 0 & 1 & 0 & 0\\
0 & 0 & 0 & 0 & 0 & 1 & 0\\
0 & 0 & 0 & 0 & 0 & 0 & 1\\
\end{array}\right)
\qquad \text{and}\qquad \lambda=
\left(\begin{array}{ccccccc}
0 & 0 & 0 & 0 & 0 & 0 & x\\
0 & 0 & 0 & 0 & 0 & 0 & y\\
0 & 0 & 0 & 0 & 0 & 0 & 0\\
0 & 0 & 0 & 0 & z & 0 & 0\\
0 & 0 & 0 & 0 & 0 & 0 & 0\\
0 & 0 & 0 & 0 & 0 & 0 & 0\\
0 & 0 & 0 & 0 & 0 & 0 & 0\\
\end{array}\right),$$
so the decorated superimposed matrix of $u-1$ and $\lambda$ is
$$\left(\ \xy<0cm,1.7cm>\xymatrix@R=.3cm@C=.3cm{
*{0} & *{0} & *{0} & *{0} & *{a} \ar @{:} @(ul,d) [ddd] & *{0} & *{x}\ar @{-} [d]\\
*{0} & *{0} & *{0} & *{0} & *{0} & *{d} \ar @{:|} [r]+<.25cm,0cm> & *{y}\ar @{-|} [lll]-<.2cm,0cm>\\
*{0} & *{0} & *{0} & *{e} & *{0} & *{0} & *{0}\\
*{0} & *{0} & *{0} & *{0} & *{b\hspace{-.05cm}z} \ar @{:} [r] & *{c}\ar @{:} [uu] & *{0}\\
*{0} & *{0} & *{0} & *{0} & *{0} & *{0} & *{0}\\
*{0} & *{0} & *{0} & *{0} & *{0} & *{0} & *{0}\\
*{0} & *{0} & *{0} & *{0} & *{0} & *{0} & *{0}
}\endxy\ \right)$$
All the connected components are pairwise compatible,
$$
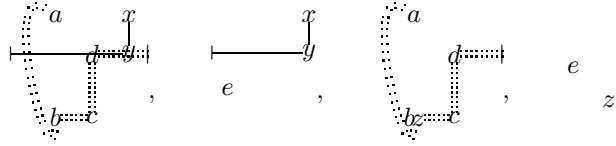
\xy<0cm,1cm>\xymatrix@R=.3cm@C=.3cm{
 *{} & *{a} \ar @{:} @(ul,d) [ddd] & *{} & *{x}\ar @{-} [d]\\
 *{} & *{} & *{d} \ar @{:|} [r]+<.25cm,0cm> & *{y}\ar @{-|} [lll]-<.2cm,0cm>\\
 *{} & *{} & *{} & *{}\\
*{} & *{b} \ar @{:} [r] & *{c}\ar @{:} [uu] & *{}
}\endxy,\qquad 
\xy<0cm,1cm>\xymatrix@R=.3cm@C=.3cm{
 *{} & *{} & *{} & *{x}\ar @{-} [d]\\
 *{} & *{} & *{} & *{y}\ar @{-|} [lll]-<.2cm,0cm>\\
 *{e} & *{} & *{} & *{}\\
*{} & *{} & *{} & *{}
}\endxy,\qquad
\xy<0cm,1cm>\xymatrix@R=.3cm@C=.3cm{
 *{} & *{a} \ar @{:} @(ul,d) [ddd] & *{} & *{}\\
 *{} & *{} & *{d} \ar @{:|} [r]+<.25cm,0cm> & *{}\\
 *{} & *{} & *{} & *{}\\
*{} & *{b\hspace{-.05cm}z} \ar @{:} [r] & *{c}\ar @{:} [uu] & *{}
}\endxy,\qquad
\xy<0cm,1cm>\xymatrix@R=.3cm@C=.3cm{
 *{} & *{} & *{} & *{}\\
 *{} & *{} & *{} & *{}\\
 *{e} & *{} & *{} & *{}\\
*{} & *{z}  & *{} & *{}
}\endxy$$
Thus, 
$$\chi^\lambda(u)=0$$
if and only if $\bag(d)\bag(y)=1$ if and only if $d(-c^{-1})b(-a^{-1})y(-x^{-1})=1$.  If $\chi^\lambda(u)\neq 0$, then
$$\chi^\lambda(u)=q^4\theta(bz)\frac{1}{q^3} \cdot 1.$$

\section{Interpolating from $U_n$ to $U_{n-1}$}

This section uses a restriction rule from $U_{(m-1)}$ to $U_{(m)}$ for supercharacters to deduce a restriction rule from $U_{n}$ to $U_{n-1}$.

\subsection{A restriction rule for $U_{(m)}$}

Note that if $\lambda\in\cZ_{(m-1)}$, then $\lambda\in \cZ_{(m)}$ unless $\cend(\lambda)=m$.  Let $\lambda^\downarrow\in \cZ_{(m)}$ be given by
$$\lambda^\downarrow=\left\{\begin{array}{ll} \lambda, & \text{if $\cend(\lambda)\neq m$,}\\ \lambda-\lambda_{im}e_{im} & \text{if $\cend(\lambda)=m$ and $i$ is minimal so $\lambda_{im}\neq 0$.}\end{array}\right.$$
Then $\lambda^\downarrow$ is the path orbit representative of $U_{(m)}\lambda U_{(m)}$.

Similarly, if $\lambda\in\cT_{(m-1)}$, then $\lambda\in \cT_{(m)}$ unless $\cend(\lambda)=m$.  Let $\lambda^\downarrow\in \cT_{(m)}$ be given by
$$\lambda^\downarrow=\left\{\begin{array}{ll} \lambda, & \text{if $\cend(\lambda)\neq m$,}\\ \lambda-\lambda_{im}e_{im} & \text{if $\cend(\lambda)=m$, $wt(\lambda)=1$, and $\lambda_{im}\in S_\lambda$,}\\
\lambda-\lambda_{im}e_{im}+\lambda_{jm}e_{jm}-\lambda_{jl}e_{jl}, & \text{if $\cend(\lambda)=m$, $wt(\lambda)=0$,  $\lambda_{im},\lambda_{jl}\in S_\lambda$ with $j=\br(\lambda)$}.\end{array}\right.$$
Then $\lambda^\downarrow$ is the comb orbit representative of $U_{(m)}\lambda U_{(m)}$.

\begin{theorem} \label{InterpolatingRestriction} Let $\lambda\in  \cZ_{(m-1)}$ and $k=\cend(\lambda)$.  Then
\begin{equation}\label{RestrictionRule} \Res^{U_{(m-1)}}_{U_{(m)}}(\chi^\lambda)=\left\{\begin{array}{ll} 
\chi^{\lambda^\downarrow}, & \text{if $k=m$ or $\wt(\lambda)=0$,}\\
q\chi^{\lambda^\downarrow}, & \text{if $k>m$, $\wt(\lambda)=1$, $\lambda_{mj}\neq 0$ for some $j>k$,}\\
\dd\sum_{t\in \FF_q}\chi^{\lambda^\downarrow+te_{mk}}, & \text{if $k>m$, $\wt(\lambda)=1$, $\lambda_{mj}=0$ for all $j>k$.}
\end{array}\right.\tag{$\ast$}
\end{equation}
\end{theorem}

\begin{proof}  Using Lemma \ref{DecomposingCharacterConjugacyComponents}, we may assume that $G_{u-1}$ has one connected component.  
We will use Theorem \ref{PathCharacterFormula}  and Theorem \ref{TriangleCharacterFormula} to compare the appropriate character values.  We split the proof into three cases.

\begin{case} The element $u\in U_{(m)}$ satisfies $u_{mk'}=0$ for all $k'>m$.
\end{case}
First, note by inspection that
$$\Res_{U_{(m)}}^{U_{(m-1)}}(\chi^\lambda)(1)=\left\{\begin{array}{ll} \chi^{\lambda^\downarrow}(1), & \text{if $k=m$ or $\wt(\lambda)=0$,}\\
q\chi^{\lambda^\downarrow}(1), & \text{if $k>m$, $\wt(\lambda)=1$, $\lambda_{mj}\neq 0$ for some $j>k$,}\\
\dd\sum_{t\in \FF_q}\chi^{\lambda^\downarrow+te_{mk}}(1), & \text{if $k>m$, $\wt(\lambda)=1$, $\lambda_{mj}=0$ for all $j>k$.}
\end{array}\right.$$

Since by assumption row $m$ of $u-1$ is zero, the fact that $1\nleq m$ does not affect the value of $\chi^\lambda(u)$.  Thus, the only problematic case is the case $\cend(\lambda)=m$.  Here observe that if $\lambda^\downarrow=\lambda-\lambda_{im}e_{im}$, then $u_{im}\neq 0$ if and only if $\chi^\lambda(u)=0=\chi^{\lambda^\downarrow}(u)$.

\begin{case} The element $u\in U_{(m)}$ satisfies $u_{mk'}\neq 0$  for some $k'>m$, and $u_{1k'}=u_{mj}=0$ for all $k'>j>m$. 
\end{case}

Note that if $k=m$ or $\wt(\lambda)=0$, then
$$\Res_{U_{(m)}}^{U_{(m-1)}}(\chi^\lambda)(u)=\chi^{\lambda^\downarrow}(u).$$
Thus, it suffices to consider the cases where $\cend(\lambda)>m$ and $\wt(\lambda)=1$.  

Suppose $\cend(\lambda)>m$ and $\wt(\lambda)=1$ with $\lambda_{mj}\neq 0$ for some $j>\cend(\lambda)$.  If $k'<j$, then $\chi^\lambda(u)=0=q\chi^{\lambda^\downarrow}(u)$.  If $k'\geq j$, then  $\chi^\lambda(u)=q\chi^{\lambda^\downarrow}(u)$ follows from the proof of Claim 1.

Suppose $k=\cend(\lambda)>m$ and $\wt(\lambda)=1$ with $\lambda_{mj}= 0$ for all $j>k$.  If $k>k'$, then
$$\sum_{t\in \FF_q}\chi^{\lambda^\downarrow+te_{mk}}(u)=\chi^{\lambda^\downarrow+0e_{mk}}(u),$$
since the other summands all contradict (PC1).  Thus, in this case, 
$$\chi^\lambda(u)=\sum_{t\in \FF_q}\chi^{\lambda^\downarrow+te_{mk}}(u).$$  
If $k<k'$, then $\chi^{\lambda^\downarrow+te_{mk}}(u)=\chi^\lambda(u)$ for all $t\in \FF_q$, so 
$$\chi^\lambda(u)=\sum_{t\in \FF_q}\chi^{\lambda^\downarrow+te_{mk}}(u).$$  
If $k=k'$, then $\chi^\lambda(u)=0$, by (PC1).  On the other hand,
$$\sum_{t\in \FF_q}\chi^{\lambda^\downarrow+te_{mk}}(u)=\frac{\chi^{\lambda^\downarrow}(1)}{q^{\#\left\{\begin{array}{@{}c@{}}\ss \text{corners strictly}\\ \vspace{-.15cm} \ss \text{NE of $(m,k')$}\end{array}\right\}}}\sum_{t\in \FF_q} \theta(tu_{mk'})=0,$$  
as desired.  

\begin{case} The element $u\in U_{(m)}$ satisfies $u_{1k'},u_{mk'}\in \FF_q^\times$ for some $k'$.
\end{case}
For Case 3 we translate to comb representatives and use Theorem \ref{TriangleCharacterFormula}.  For comb representatives, we may assume that $u_{1k'},u_{mk'}$ are the only nonzero entries in $u-1$.  Define $u^{\uparrow}\in U_{(m-1)}$ by $u^\uparrow=u-u_{1k'}e_{1k'}$.  Note that $u^\uparrow$ is the representative for the superclass containing $u$ in $U_{(m-1)}$.  Thus, we will show
$$\Res_{U_{(m)}}^{U_{(m-1)}}(\chi^{\lambda[u^\uparrow,k']})(u^\uparrow)=\left\{\begin{array}{ll} \chi^{\lambda^\downarrow[u,k']}(u), & \text{if $k=m$ or $\wt(\lambda)=0$,}\\
q\chi^{\lambda^\downarrow[u,k']}(u), & \text{if $k>m$, $\wt(\lambda)=1$, $\lambda_{mj}\neq 0$ for some $j>k$,}\\
\dd\sum_{t\in \FF_q}\chi^{\lambda^\downarrow[u,k']+te_{ml}}(u), & \text{if $k>m$, $\wt(\lambda)=1$, $\lambda_{1l}\neq 0$, $\lambda_{mj}=0$, $j>k$.}\end{array}\right.$$
In the cases $\wt(\lambda)=0$ or $k=m$, the character value $\chi^{\lambda[u^\uparrow,k']}(u^\uparrow)=0$ if and only if at least one of the following conditions hold
\begin{enumerate}
\item[(a)] $\lambda_{ij}\neq 0$ for $1<j<k'$,
\item[(b)] $\lambda_{ml}\neq 0$ for some $l>k'$,
\end{enumerate}
if and only if $\chi^{\lambda^\downarrow[u,k']}(u)=0$.  The restrictions for these cases follow.  

Suppose $\cend(\lambda)>m$, $\wt(\lambda)=1$ and $\lambda_{mj}\neq 0$ for some $j>k=\cend(\lambda)$.  If $j>k'$, then $\chi^{\lambda[u^\uparrow,k']}(u^\uparrow)=0=\chi^{\lambda^\downarrow[u,k']}(u)$.  If $j\leq k'$, then $k<k'$, and $\chi^{\lambda[u^\uparrow,k']}(u^\uparrow)=0=\chi^{\lambda^\downarrow[u,k']}(u)$ by a similar argument as in the previous cases.

Suppose $\cend(\lambda)>m$ $\wt(\lambda)=1$ and $\lambda_{mj}= 0$ for all $j>k$.  Let $l$ be such that $\lambda_{1l}\neq 0$.  If $l<k'$, then $\chi^{\lambda^\downarrow[u,k']+te_{ml}}(u)=\chi^{\lambda^\downarrow[u,k']}(u)=q^{-1}\chi^{\lambda[u^\uparrow,k']}(u)$.  If $l=k'$, then
$\chi^{\lambda[u^\uparrow,k']}(u)=0=\chi^{\lambda^\downarrow[u,k']+te_{ml}}(u)$.  

Suppose $l>k'$.  If $\lambda_{jk'}\neq 0$ for some $1<j<k'$, then  $\chi^{\lambda[u^\uparrow,k']}(u)=0$. On the other hand, by Theorem \ref{TriangleCharacterFormula},
$$\sum_{t\in \FF_q}\chi^{\lambda^\downarrow[u,k']+te_{ml}}(u)=\frac{\chi^\lambda(1)}{q^{|C_{k'}|}}+\sum_{t\in \FF_q^\times}\frac{\chi^\lambda(1)}{q^{|C_{k'}|}}\theta\bigl(-\lambda_{jl}^{-1}\lambda_{jk'}(u_{1k'}\lambda_{1l}+u_{mk'}t)\big)=0.$$ 
If $\lambda_{jk'}=0$ for all $j<k'$, then
$$\sum_{t\in \FF_q}\chi^{\lambda^\downarrow[u,k']+te_{ml}}(u)=\chi^{\lambda^\downarrow[u,k']-u_{1k'}\lambda_{1l}u_{mk'}^{-1}e_{ml}}(u),$$
since, according to condition (1) of Theorem \ref{TriangleCharacterFormula}, all the other summands are zero.  However, $\chi^{\lambda^\downarrow[u,k']-u_{1k'}\lambda_{1l}u_{mk'}^{-1}e_{ml}}(u)=\chi^{\lambda[u^\uparrow,k']}(u)$, as desired.
\end{proof}

\subsection{A restriction rule for $U_n$}

For $\lambda\in \cS_{n}(q)$, and $i,k\in \ZZ_{\geq 1}$, let
\begin{equation}\label{RightProduct}
\lambda\ast_{i}\{k\}=\left\{\begin{array}{ll} 
\lambda, & \text{if $i=k$,}\\ 
q\lambda\ast_{i+1}\{k\}, & \text{if $\lambda_{il}\neq 0$ for some $l>k$,}\\
\lambda\big|_{\lambda_{ik}=0}\ast_{i+1} \{k\}, & \text{if $\lambda_{ik}\neq 0$,}\\
\dd\lambda\ast_{i+1} \{k\} +\sum_{t\in \FF_q^\times} \lambda\big|_{\lambda_{ik}=t\atop\lambda_{ij}=0}\ast_{i+1}\{j\}, & \text{if $\lambda_{ij}\neq 0$ for some $i<j<k$,}\\
\dd\lambda\ast_{i+1}\{k\}+\sum_{t\in \FF_q^\times}\lambda\big|_{\lambda_{ik}=t}, & \text{if $\lambda_{ij}=0$ for all $j>i$.}
\end{array}\right.
\end{equation}
We will extend this product to supercharacters by the convention
\begin{equation*}
\chi^\mu\ast_i \chi^{\{k\}} =\sum_{\lambda} c_{\mu k}^\lambda \chi^\lambda,\qquad\text{if} \qquad \mu\ast_i \{k\}=\sum_{\lambda}c_{\mu k}^\lambda\lambda.
\end{equation*}

\begin{corollary} \label{URestrictionRule}
For $\lambda\in \cS_n(q)$,
$$\Res_{U_{n-1}}^{U_n}(\chi^\lambda)=\left\{\begin{array}{ll} \chi^\lambda\ast_{1}\chi^{\{k\}}, & \text{if $\lambda_{1k}\neq 0$, for some $1<k$,}\\ \chi^\lambda, & \text{otherwise.}\end{array}\right.$$ 
\end{corollary}
\begin{proof}
Note that for $1\leq m\leq n$, Theorem \ref{InterpolatingRestriction} implies
$$\Res^{U_{(m-1)}}_{U_{(m)}}(\chi^\lambda)=\left\{\begin{array}{ll} \chi^{\lambda}, & \text{if $\lc(\lambda)\neq m$ and $\wt(\lambda)=0$,}\\
\chi^{\lambda-\lambda_{im}e_{im}}, & \text{if $\lc(\lambda)=m$,} \\ q\chi^{\lambda}, & \text{if $\lc(\lambda)>m$, $\wt(\lambda)=1$, $\lambda_{mj}\neq 0$ for some $j>\lc(\lambda)$,}\\
\dd\sum_{t\in \FF_q}\chi^{\lambda+te_{mk}}, & \text{if $k=\lc(\lambda)>m$, $\wt(\lambda)=1$, $\lambda_{mj}=0$, for all $j>k$.}\end{array}\right.$$
The multiplication given by (\ref{RightProduct}) is an iterative version of the restrictions from $U_{(m-1)}$ to $U_{(m)}$, where the last two cases in (\ref{RightProduct}) correspond to the last case in the restriction, depending on whether there exists  $\lambda_{mj}\in \FF_q^\times$ for some $j<\lc(\lambda)$.
\end{proof}

\subsection{Examples}

\noindent \textbf{Example 1.} Consider the case $q=2$.  Then we may choose our representatives $\lambda\in \fkn_n^*$ and $u\in U_n$ to correspond to set partitions of $\{1,2,\ldots, n\}$.  For example, if 
$$\lambda=\{1\larc{}5\mid 2\larc{} 6\mid 3\larc{} 4\},$$
then
\begin{align*}
\Res^{U_6}_{U_{5}}(\chi^\lambda) &= \chi^{\{1\larc{}5\mid 2\larc{} 6\mid 3\larc{} 4\}}\ast_1\{5\}\\
						     &=\chi^{\{1\mid 5\mid 2\larc{} 6\mid 3\larc{} 4\}}\ast_2\{5\}\\
						     &=2\chi^{\{1\mid 5\mid 2\larc{} 6\mid 3\larc{} 4\}}\ast_3\{5\}\\
						     &=2\chi^{\{1\mid 5\mid 2\larc{} 6\mid 3\larc{} 4\}}\ast_4\{5\}+2\chi^{\{1\mid 2\larc{} 6\mid 3\larc{} 5\mid 4\}}\ast_4\{4\}\\
&=2\chi^{\{1\mid 5\mid 2\larc{} 6\mid 3\larc{} 4\}}\ast_5\{5\}+2\chi^{\{1\mid 2\larc{} 6\mid 3\larc{} 4\larc{} 5\}}\ast_5\{5\}+2\chi^{\{1\mid 2\larc{} 6\mid 3\larc{} 5\mid 4\}}\\
&=2\chi^{\{1\mid 5\mid 2\larc{} 6\mid 3\larc{} 4\}}+2\chi^{\{1\mid 2\larc{} 6\mid 3\larc{} 4\larc{} 5\}}+2\chi^{\{1\mid 2\larc{} 6\mid 3\larc{} 5\mid 4\}}\\
\end{align*}
Note that the final result gives the representative of $U_{n-1}$ as a submatrix of $U_n$ (since by construction the first row will always be zero).  To obtain a set partition of $n-1$, we would remove $1$ and renumber the rest of the entries $j\mapsto j-1$.

Alternatively, this algorithm may be viewed as a ``bumping algorithm", where we replace the $1$ by all other ``possibilities," suitably defined.  

\vspace{.25cm}

\noindent \textbf{Example 2.}  Linear characters of $U_n$ are those characters whose superclass representative satisfies, $i$ and $j$ are in the same part only if $i+1, i+2,\ldots, j-1$ are also in the part.  In this case, 
$$\chi^\lambda\ast_1\{j\}=\chi^\lambda.$$

\vspace{.25cm}

\noindent\textbf{Example 3.}  On the opposite extreme with have the case
$$\lambda=\{1\larc{} n\mid 2\mid 3\mid\cdots\mid n-1\}.$$
In this case,
$$\chi^\lambda\ast_1\{n\}=\sum_{\mu\in \cS_{n}(q),\atop\lambda_{ij}=0, 1\leq i<j<n}\chi^\mu.$$

\subsection{An alternate embedding of $U_{n-1}$}

The paper \cite{MT07} uses  a different embedding of $U_{n-1}$ into $U_n$ (obtained by removing the last column rather than the first row).  This alternate embedding gives a different restriction rule.  For $\mu\in \cS_n(q)$ and $j,l\in \ZZ_{\geq 1}$, let
$$\{j\} \ast_l \mu=\left\{\begin{array}{ll} \mu, & \text{if $j=l$,}\\ 
q (\{j\} *_{l-1} \mu), & \text{if there is $i<j$ with $\mu_{il}\neq 0$,}\\
\dd \{j\} \ast_{l-1} \mu\big|_{\mu_{jl}=0}, & \text{if $\mu_{jl}\neq 0$,}\\
\dd \{j\}\ast_{l-1}\mu+\sum_{t\in \FF_q^\times} \{k\} \ast_{l-1} \mu\big|_{\mu_{kl}=0\atop\mu_{jl}=t}, & \text{if there is $k>j$ with $\mu_{kl}\neq 0$,}\\
\dd \{j\} \ast_{l-1} \mu + \sum_{t\in \FF_q^\times} \mu\big|_{\mu_{jl}=t}, & \text{otherwise.}\end{array}\right.
$$
Then by symmetry arguments from \cite{MT07}, we obtain the following corollary for this alternate embedding of $U_{n-1}$ into $U_n$.

\begin{corollary} 
For $\lambda\in \cS_n(q)$,
$$\Res_{U_{n-1}}^{U_n}(\chi^\lambda)=\left\{\begin{array}{ll} \chi^{\{j\}} \ast_n \chi^\lambda, & \text{if $\lambda_{jn}\neq 0$, for some $j<n$,}\\ \chi^\lambda, & \text{otherwise.}\end{array}\right.$$ 
\end{corollary}

In particular, unlike in the symmetric group representation theory, the decomposition of induced characters depends on the embedding of $U_{n-1}$ into $U_n$.

\end{document}